\documentclass[twoside]{amsart}
\usepackage{latexsym}
\usepackage{amssymb}
\usepackage{amsmath}
\usepackage{amscd}
\usepackage{mathrsfs}
\usepackage[all]{xypic}
\usepackage{color}

\newtheorem{theorem}{Theorem}[section]
\newtheorem{lemma}[theorem]{Lemma}
\newtheorem{corollary}[theorem]{Corollary}
\newtheorem{proposition}[theorem]{Proposition}

\theoremstyle{definition}

\newtheorem{definition}[theorem]{Definition}
\newtheorem{example}[theorem]{Example}

\theoremstyle{remark}

\newtheorem{remark}[theorem]{Remark}

\DeclareMathOperator{\Aut}{{\rm Aut}}

\DeclareMathOperator{\dfn}{:=}

\DeclareMathOperator{\depth}{{\sf depth}}
\DeclareMathOperator{\isoms}{\sim_{cc}}  %isomorphic modulo finite subgroup

\newcommand{\con}[1]{\mathop{\sf con}({#1})}
\newcommand{\bcon}[1]{\mathop{\sf bco}({#1})}

\newcommand{\rbcon}[2]{\mathop{\sf rbco}({#1},{#2})}
\newcommand{\ret}[1]{\mathop{\sf ret}({#1})}
\newcommand{\nub}[1]{\mathop{\sf nub}({#1})}
\newcommand{\bcg}[1]{\textstyle \mathop{\sf con}_{\leftrightarrow}({#1})}
\newcommand{\open}[1]{{\mathscr #1}}

\def\ident{1}
\def\triv{\{\ident\}}
\def\NN{\mathbb N}
\def\ZZ{\mathbb Z}
\def\Cp{C_p} %the cyclic group of order p

\def\endproof{\hfill$\square$}

\begin{document}
\bibliographystyle{plain}

\title[Contraction Groups]{The Nub of an Automorphism of a Totally Disconnected, Locally Compact Group}

\author[G. A. Willis]{George A. Willis}
\thanks{Research supported by ARC Discovery Grant DP0984342}
\address{Department of Mathematics \\ 
University of Newcastle\\ Callaghan, NSW 2308, Australia}
\email{george.willis@newcastle.edu.au}


\subjclass{Primary:  Secondary: } 
 
\begin{abstract}
To any automorphism, $\alpha$, of a  totally disconnected, locally compact group, $G$, there is associated a compact, $\alpha$-stable subgroup of $G$, here called the \emph{nub} of $\alpha$, on which the action of $\alpha$ is topologically transitive. Topologically transitive actions of automorphisms of compact groups have been studied extensively in topological dynamics and results obtained transfer, via the nub, to the study of automorphisms of general locally compact groups.

A new proof that the contraction group of $\alpha$ is dense in the nub is given, but it is seen that the two-sided contraction group need not be dense. It is also shown that each pair $(G,\alpha)$, with $G$ compact and $\alpha$ topologically transitive, is an inverse limit of pairs that have `finite depth' and that analogues of the Schreier Refinement and Jordan-H\"older Theorems hold for pairs with finite depth. 
\end{abstract}

\maketitle

\section{Introduction}
\label{sec:intro}
The automorphisms of a group reveal aspects of its structure, and the present paper extends previous work in \cite{BaumW_Cont,Wi:structure,wi:further,Wi:SimulTriang} on automorphisms of totally disconnected, locally compact groups. Associated with each automorphism, $\alpha$, of the totally disconnected, locally compact group, $G$, there is a compact, $\alpha$-stable subgroup of $G$ here called the \emph{nub} of $\alpha$. This paper has two aims: to answer questions from \cite{BaumW_Cont} and \cite{Wi:SimulTriang} about the relationship between the nub subgroup and other features of $\alpha$; and to investigate the structure of the nub. Results and examples from topological dynamics found in \cite{KSchmidt} transfer and assist in meeting these aims. 

Concerning the first aim, it is seen in \cite{BaumW_Cont} that the contraction group corresponding to $\alpha$ is closed if and only if the nub of $\alpha$ is trivial. Similarly, the second step in the procedure for finding a subgroup tidy for $\alpha$, described in \cite{Wi:structure,wi:further,Wi:SimulTriang}, is  automatic if the nub is trivial. Section~\ref{sec:nub} below describes the relationship between the nub of $\alpha$ and the tidying procedure more fully. Several equivalent characterizations of the nub subgroup are also given in that section. 

A question left unresolved in \cite{BaumW_Cont} and \cite{Wi:SimulTriang} was whether the `two-sided contraction group', that is, $\left\{ x\in G \mid \alpha^n(x)\to\ident_G \hbox{ as } |n|\to\infty\right\}$ is dense in the nub of $\alpha$. A positive answer to that question would simplify matters in those papers and in Section~\ref{sec:nub} of this paper. In an attempt to answer the question, the author investigated the structure of the nub of $\alpha$ in more detail and the results of that investigation are reported in Sections~\ref{sec:contraction} and~\ref{sec:Jordan-Holder}. One of the outcomes is a new proof that the (one-sided) contraction group is dense in the nub: where the original proof in \cite{BaumW_Cont} relies on a topological argument, the proof in Section~\ref{sec:contraction} places more reliance on the algebraic structure. Another outcome may be seen in Theorems~\ref{thm:Schreier} and~\ref{thm:Jordan-Holder}, which are a type of Schreier Refinement and Jordan-H\"older Theorem for pairs $(G,\alpha)$, where $G$ is a compact, totally disconnected group and $\alpha\in\Aut(G)$. Unfortunately (perhaps), it turned out that the `two-sided' contraction group is not always dense in the nub subgroup, as Example~\ref{ex:con_2-ways_not_dense} shows.

The short Section~\ref{sec:alphastable} applies the structure theorems for the nub of $\alpha$ obtained in earlier sections to produce an abstract version of a technical result of J.~Tits concerning automorphism groups of trees.

As pointed out to me by Klaus Schmidt, much of the work in this paper, including Example~\ref{ex:con_2-ways_not_dense}, reproduces results that are already well-known in the topological dynamics literature, see the book \cite{KSchmidt} for results known prior to 1995 and \cite{Fagnani,KitchSchmidt2,Lind&Schmidt} for  more recent work. The main contributions made here appear to be: the link made with the structure theory of non-compact, totally disconnected, locally compact groups; sharper and more algebraic formulations of some results, such as Theorems~\ref{thm:Schreier} and~\ref{thm:Jordan-Holder} vis-\`a-vis \cite[Proposition~10.2]{KSchmidt}; and a different perspective and new methods of proof. Where possible, terms from the dynamical systems literature have been adopted. For example, the `two-sided contraction group' is called the \emph{homoclinic subgroup}. However, there is one important difference. Automorphisms that are called `expansive' in the dynamical systems literature are here called `finite depth': the term `expansive' conveys a misleading intuition from the point of view of the group structure theory while, on the other hand, `finite depth' suggests the analogy that motivates Theorems~\ref{thm:Schreier} and~\ref{thm:Jordan-Holder}. 

The next two Sections~\ref{sec:keyconcepts} and~\ref{sec:Associates} review the main concepts from the structure theory of totally disconnected, locally compact groups and from topological dynamics that will be needed in the rest of the paper.   

\section{Automorphisms in Structure Theory and Ergodic Theory}
\label{sec:keyconcepts}

The results in this paper are at the interface between the structure theory of totally disconnected, locally compact groups and the theory of ergodic actions by group automorphisms, and relevant concepts from these theories are recalled here. Key examples that may be used to illustrate ideas in both theories are defined first.
 
\subsection{Product groups and shifts}
\label{sec:examples}
For each finite group $F$,  
$$
F^{\mathbb{Z}} \dfn \left\{ f : \mathbb{Z} \to F \right\}
$$ 
is the \emph{direct product} over $\mathbb{Z}$. Equipped with the coordinatewise product and product topology, $F^{\mathbb{Z}}$ is a compact, totally disconnected group. The following particular elements of $\mathbb{Z}$ will be referred to frequently. \\
\emph{Point support:} For $f\in F$ and $n\in \mathbb{Z}$, denote by $[f]_{n}$ the element of $F^{\mathbb{Z}}$ such that
$$
[f]_{n}(k) = \begin{cases}
f & \hbox{ if } k=n\\
\ident_F & \hbox{ if } k\ne n
\end{cases}.
$$
\emph{Constants:} The constant function on $\mathbb{Z}$ with value $f$ will be denoted by $\underline{f}$ and the subgroup of constant functions, which is isomorphic to $F$, will be denoted by $\underline{F}$. 
\\
The \emph{shift automorphism}, $\sigma$, of $F^{\mathbb{Z}}$ is defined by
$$
\sigma(f)(n) = f(n+1)
$$
and the pair $(F^{\mathbb{Z}},\sigma)$ will be called the \emph{shift}.

In many examples, $F$ is the group $\left\{ \bar{i} \mid i\in\{0,1,\dots, n-1\}\right\}$ of integers under addition modulo $n$. This group is isomorphic to the cyclic group of order $n$ and will be denoted $C_n$. The ring structure on the integers modulo $n$ will be used on occasion, and $C_n$ will then denote this ring (which is a field when $n$ is prime). 

Two $\sigma$-stable subgroups of $F^{\mathbb{Z}}$ will be referred to frequently.  \\
The \emph{direct sum} $\bigoplus_{\mathbb{Z}} F$ will be identified with the subgroup of $F^{\mathbb{Z}}$ consisting of functions with finite support and denoted by $F^{[\mathbb{Z}]}$. Then $F^{[\mathbb{Z}]}$ is generated by the elements $f_{[n]}$ ($f\in F$, $n\in\mathbb{Z}$) and is a dense subgroup of $F^{\mathbb{Z}}$. \\
The \emph{restricted sum} $(F^{\mathbb{Z}^-})\oplus (\bigoplus_{\mathbb{Z}\cup\{0\}} F)$ 
will be identified with the subgroup of $F^{\mathbb{Z}}$ consisting of functions whose support is contained in $(-\infty,n]$ for some $n\in\mathbb{Z}$ and denoted by $F^{\mathbb{Z}<}$. Then $F^{\mathbb{Z}<}$ is a dense subgroup of $F^{\mathbb{Z}}$. It is a non-compact, totally disconnected, locally compact group when equipped with the topology in which all subgroups $F^{(-\infty,n]}$ are compact and open. 

It will often be convenient to consider pairs $(G,\alpha)$, where $G$ is a compact group and $\alpha\in\Aut(G)$, as objects in a category in which morphisms, $\phi: (G,\alpha)\to (H,\beta)$, are  group homomorphisms, $\phi: G\to H$, that intertwine $\alpha$ and $\beta$. It is shown in later sections how general objects in this category may be described in terms of shifts $(F^{\mathbb{Z}},\sigma)$.

\subsection{Totally disconnected, locally compact groups and the scale}
\label{sec:minimizing,scale}
Each totally disconnected, locally compact group has a base of neighbourhoods of the identity consisting of compact, open subgroups, see \cite{vDant,HewRoss}. Such subgroups will be denoted $V$, $W$, \dots\ and 
$$
{\mathcal B}(G) := \left\{ V \mid V\leq G \hbox{ is compact and open}\right\}.
$$
Since an open subgroup of a compact group has finite index, any two subgroups in ${\mathcal B}(G)$ are \emph{commensurable}, that is, their intersection has finite index in both. This allows a positive integer to be associated with each automorphism of $G$, as follows. 
\begin{definition}
\label{defn:scale&minim}
The \emph{scale} of the automorphism $\alpha$ is the positive integer
$$
s(\alpha) := \min\left\{ [\alpha(V) : \alpha(V)\cap V] \mid V\in {\mathcal B}(G)\right\}.
$$ 
A subgroup, $V$, at which the minimum is attained is \emph{minimizing} for~$\alpha$. 
\end{definition}
\noindent When $G$ has a compact, open, normal subgroup all inner automorphisms have scale~1. The scale can be a useful tool for investigating general totally disconnected, locally compact groups that do not have such a subgroup, see \cite{RoWi:conclcg,Pre+Wu:dense, Wi:tdHM}.

Applications of the concepts of scale and minimizing subgroup, in \cite{DaShWi:Lie,RoWi:conclcg} for instance, depend on the following structural characterization of minimizing subgroups, which is proved in \cite{Wi:structure,wi:further}.

\begin{theorem}[The Structure of Minimizing Subgroups]
\label{thm:tidiness}
Let $\alpha$ be an automorphism of $G$. For each $V\in {\mathcal B}(G)$ put
$$
V_+ = \bigcap_{k\geq0} \alpha^k(V) \ \hbox{ and }\ V_- = \bigcap_{k\geq0} \alpha^{-k}(V).
$$
Then $V$ is minimizing for $\alpha$ if and only if
\begin{description}
\item[TA($\alpha$)] $V = V_+V_-$ and 
\item[TB($\alpha$)] $V_{++} := \bigcup_{k\geq0} \alpha^k(V_+)$ and $V_{--} := \bigcup_{k\geq0} \alpha^{-k}(V_-)$ are closed.
\end{description}
If $V$ is minimizing for $\alpha$, then $s(\alpha) = [\alpha(V_+):V_+]$. 
\endproof
\end{theorem}
\noindent A subgroup satisfying $\text{\bf TA}(\alpha)$ is said to be \emph{tidy above} for $\alpha$ and one satisfying $\text{\bf TB}(\alpha)$ is \emph{tidy below}. A subgroup satisfying $\text{\bf TA}(\alpha)$ and $\text{\bf TB}(\alpha)$ is \emph{tidy} for $\alpha$ so that the theorem may be restated as \emph{$V$ is minimizing for $\alpha$ if and only if it is tidy for $\alpha$}. The proof of Theorem~\ref{thm:tidiness} uses a certain procedure for finding subgroups tidy for $\alpha$ that is discussed further in Remark \ref{rem:TBfirst}.

To illustrate these ideas, the only subgroup of $F^{\mathbb{Z}}$ that is minimizing for $\sigma$ is the whole group and $s(\sigma) = 1$. On the other hand, all compact, open subgroups $F^{(-\infty,n]}$ of $F^{\mathbb{Z}<}$ are minimizing for the shift automorphism $\sigma$. Putting $V = F^{(-\infty,n]}$, we have $V_+=\triv$, $V_- = V$, $s(\sigma) = 1$ and $s(\sigma^{-1}) = |F|$. 

The following result is \cite[Lemma~1]{Wi:structure}. It is a key step in the proof of Theorem~\ref{thm:tidiness} above and will be applied in later sections to automorphisms of compact groups. \begin{lemma}
\label{lem:recalltidy}
Let $G$ be a totally disconnected, locally compact group and let $\alpha$ be an automorphism. Suppose that $V$ is a compact, open subgroup of $G$. Then:
\begin{enumerate}
\item there is $n\in \mathbb{N}$ such that $\bigcap_{k=0}^n \alpha^k(V)$ is tidy above for $\alpha$, and\label{lem:recalltidyi}
\item $V_+$ has finite index in $\alpha(V_+)$. \label{lem:recalltidyii}
\end{enumerate}
\endproof
\end{lemma}

The following useful criteria for an orbit to be contained in the tidy subgroup $V$ may be found in \cite[Lemma~3.2]{wi:further}.
\begin{lemma}
\label{lem:criterion}
Let $V$ be tidy for $\alpha$ and $x\in V$. 
\begin{enumerate}
\item If $x\in V_+$ and $\{\alpha^n(x)\}_{n\in{\mathbb{N}}}$ is bounded, then $\{\alpha^n(x)\}_{n\in{\mathbb{Z}}}\leq V$.\label{lem:criterioni}
\item If $\{\alpha^n(x)\}_{n\in{\mathbb{Z}}}$ is bounded, then $\{\alpha^n(x)\}_{n\in{\mathbb{Z}}}\leq V$. \label{lem:criterionii}
\end{enumerate}
\endproof
\end{lemma}

\subsection{Ergodic, or topologically transitive, actions by automorphisms}
The action of a Borel isomorphism, $T$, on a measure space, $(X,\mu)$, is \emph{ergodic} if every $T$-stable Borel subset of $X$ has either zero or full measure. Beginning with the conjecture of Paul Halmos \cite{Halmos} that a locally compact group having an automorphism that acts ergodically for the Haar measure must be compact, there has been much interest in ergodic theory in the action of single automorphisms on locally compact groups. Halmos' conjecture was proved in \cite{KaufRaj} in the case of connected groups and in~\cite{Aoki} for totally disconnected groups. A short proof of the totally disconnected case that uses the scale of the automorphism and associated ideas is given in~\cite{Pre+Wu:dense}. An extension of Halmos' conjecture to finitely generated, abelian group of automorphisms acting ergodically is proved in \cite{DaShWi:Lie}.

With the proof of Halmos' conjecture, interest has shifted to ergodic actions of automorphisms of compact groups. The monograph \cite{KSchmidt} gives a full account of progress up to 1995. Among the numerous publications that have appeared since are \cite{Fagnani,Lind&Schmidt,KitchSchmidt2}.

The action of the shift $\sigma$ on the product $F^{\mathbb{Z}}$ is ergodic. The restriction of $\sigma$ to a closed, $\sigma$-stable subgroup, $G$, of $F^{\mathbb{Z}}$ gives rise to a \emph{subshift of finite type}, see \cite{}. Every ergodic pair $(G,\alpha)$, with $G$ compact and totally disconnected, is an inverse limit of subshifts of finite type, see \cite[Proposition~2.17]{KSchmidt} and also Proposition~\ref{prop:Proj_Limit} below. This key technique in the ergodic theoretic approach that seems to play the corresponding role to that of Lemma~\ref{lem:recalltidy} above in the algebraic approach. %Since $\alpha$ acts ergodically on its nub, this fact is also important in the algebraic structure theory and is discussed further in Section~\ref{sec:contraction}. 

Ergodicity of the pair $(G,\alpha)$ is equivalent to the condition of being \emph{topologically transitive}, that is, there is $x\in G$ such that the orbit $\left\{ \alpha^n(x) \mid n\in\mathbb{Z}\right\}$ is dense, see \cite[Theorem~1.1]{KSchmidt}. Since the emphasis in this paper is on the algebraic and topological properties of automorphisms, the latter condition will be used here. 

\section{Subgroups Associated with an Automorphism}
\label{sec:Associates}

Several subgroups of $G$ associated with the automorphism $\alpha$ have been studied in the literature. These subgroups and the relationships between them are recalled in this section and  a consistent notation for them is introduced. 
\begin{definition}
\label{defn:contraction}
Let $\alpha$ be an automorphism of the totally disconnected, locally compact group $G$. 
\begin{enumerate}
\item The \emph{contraction group} for $\alpha$ is \label{defn:contraction1}
\begin{multline*}
\phantom{AAAAAAAA}\con{\alpha} := \left\{ x\in G \mid \alpha^{n}(x) \to \ident_G \hbox{ as }n\to\infty \right\}.\hfill
\end{multline*} 
\item The \emph{retraction group} for $\alpha$ is
\begin{multline*}
\phantom{AAAAAAAA}\ret{\alpha} := \bigcap\left\{ V_{--} \mid V\in {\mathcal B}(G)\hbox{ is tidy for }\alpha\right\}.\hfill
\end{multline*}
\item The \emph{bounded contraction group} for $\alpha$ is 
\begin{multline*}
\phantom{AAAAAAAA}\bcon{\alpha} := \left\{ x\in \con{\alpha}\mid \{\alpha^n(x)\}_{n\leq0} \text{ has compact closure}\right\}.\hfill
\end{multline*} 
\item The \emph{two-sided contraction}, or \emph{homoclinic group} for $\alpha$ is \label{defn:contraction4}
\begin{multline*}
\phantom{AAAAAAAA}\bcg{\alpha} :=  \con{\alpha}\cap \con{\alpha^{-1}}.\hfill
\end{multline*} 
\item The \emph{nub group} for $\alpha$ is \label{defn:contraction5}
\begin{multline*}
\phantom{AAAAAAAA}\nub{\alpha} := \bigcap\left\{ V \mid V\in {\mathcal B}(G)\hbox{ is tidy for }\alpha\right\}.\hfill
\end{multline*}
\end{enumerate}
\end{definition}
The term homoclinic subgroup in (\ref{defn:contraction4}) comes from the ergodic theory literature, see \cite{Lind&Schmidt}, and will be adopted in this paper. 

The above definitions may be illustrated by considering the shift automorphism on $G= F^{\mathbb{Z}}$, as defined in Subsection~\ref{sec:examples}. In this case: $\con{\sigma} = F^{\mathbb{Z}<} = \bcon{\sigma}$; $\ret{\sigma} = F^{\mathbb{Z}} = \nub{\sigma}$; and $\bcg{\sigma} = F^{[{\mathbb{Z}}]}$. Consider also the shift automorphism on $G= F^{\mathbb{Z}<}$. In this case: $\con{\sigma} = F^{\mathbb{Z}<} = \ret{\sigma}$; and $\bcon{\sigma} = \bcg{\sigma} =  \nub{\alpha} = \triv$. Note that, in both of these examples, $\con{\sigma}$ is dense in $\ret{\sigma}$ and $\bcg{\sigma}$ is dense in $\nub{\sigma}$. It will be seen that the first of these statements holds for every automorphism but that the second does not. 

Several of these groups are clearly equal when $G$ is compact. 
\begin{proposition}
\label{prop:ret=nub}
Let $G$ be a compact, totally disconnected group and $\alpha$ be an automorphism of $G$. Then: $V\leq G$ is tidy for $\alpha$ if and only if it is $\alpha$-stable; $\con{\alpha} = \bcon{\alpha}$; and $\nub{\alpha} = \ret{\alpha}$. 
\endproof
\end{proposition}

\begin{remark}
{\it (a)\/} Contraction groups for an automorphism are important in representation theory, \cite{Moore, Wang} and the study of convolution semigroups, \cite{HazSie,Sie}. Their role in the structure theory of totally disconnected, locally compact groups is worked out in \cite{BaumW_Cont,GlockCont,WJaw,WJaw2}. A rather complete structure theorem for closed contraction groups, that is, when $\con{\alpha} = \ret{\alpha}$, is proved in \cite{GWClosedCont}.  \\
{\it (b)\/} The contraction group is a dense subgroup of the retraction group, as shown in \cite[Theorem~3.26]{BaumW_Cont} for metrizable groups and in~\cite{WJaw}  in the general case. The contraction group is closed, and hence equal to the retraction group, if and only if $\nub{\alpha} = \triv$, which occurs if and only if $\bcon{\alpha}  = \triv$, see~\cite[Theorem~3.32(4)]{BaumW_Cont}. \\ 
{\it (c)\/} As a special case, it is shown in~\cite[Lemma~3.31(2)]{BaumW_Cont} that $\nub{\alpha}$ is equal to the closure of $\bcon{\alpha}$ when~$G$ is metrizable. Also see Proposition~\ref{prop:con_dense} in this paper. \\
 {\it (d)\/} Lemma~3.28 in \cite{BaumW_Cont} shows that $V_{--}\cap V_{++} = V_0 \dfn \bigcap_{n\in\mathbb{Z}} \alpha^n(V)$ for every subgroup tidy for $\alpha$. Hence 
$\nub{\alpha} = \ret{\alpha}\cap\ret{\alpha^{-1}}$. It is a natural question therefore whether $\bcg{\alpha}$ is dense in $\nub{\alpha}$. Proposition~\ref{prop:con_dense}  shows that the closure of $\bcg{\alpha}$ contains the commutator subgroup of $\nub{\alpha}$, but Example~\ref{ex:con_2-ways_not_dense} shows that it may be a proper subgroup. 
\end{remark}

The following additional subgroup associated with $\alpha$ was introduced in~\cite{Wi:SimulTriang}. 
\begin{definition}
\label{defn:relative_contraction}
Let $\alpha\in \Aut(G)$ and $V\in {\mathcal B}(G)$. 
The \emph{bounded contraction group relative to $V$} is
\begin{eqnarray*}
\rbcon{\alpha}{V} &\dfn& \left\{ x\in G \mid \exists N\hbox{ such that }\alpha^n(x)\in V \hbox{ for all } n\geq N\right.\\
&& \qquad\qquad \left. \hbox{ and }\left\{\alpha^n(x)\right\}_{n\leq 0}\text{ has compact closure} \right\}.
\end{eqnarray*}
\end{definition}
\noindent This subgroup was used in an argument concerning multiple commuting automorphisms for which it was necessary to characterize property $\text{\bf TB}(\alpha)$ independently of $\text{\bf TA}(\alpha)$. This characterization, established in \cite[Proposition~2.1]{Wi:SimulTriang}, is that any compact, open subgroup containing $\bigcap\left\{ \overline{\rbcon{\alpha}{V}}\mid {V\in \mathcal{B}(G)}\right\}$ is tidy below\footnote{The subgroup $\rbcon{\alpha}{V}$ is not the same as the \emph{contraction group modulo a compact, $\alpha$-stable subgroup $K$} studied in \cite{BaumW_Cont,WJaw,WJaw2}.}. 

By definition, $\bcon{\alpha} = \bigcap\left\{{\rbcon{\alpha}{V}}\mid {V\in \mathcal{B}(G)}\right\}$. Hence
\begin{equation}
\label{eq:bcon_and_rbcon}
\overline{\bcon{\alpha}} \leq \bigcap\left\{\overline{\rbcon{\alpha}{V}}\mid {V\in \mathcal{B}(G)}\right\}.
\end{equation}
Section~\ref{sec:nub} returns to the question of characterizing property $\text{\bf TB}(\alpha)$ independently of $\text{\bf TA}(\alpha)$ and shows that both groups in (\ref{eq:bcon_and_rbcon}) are in fact equal to $\nub{\alpha}$, thus simplifying some of the arguments in \cite{Wi:SimulTriang}. 

\section{The Nub of an Automorphism and Tidy Subgroups}
\label{sec:nub}

Compact subgroups of $G$ that are stable under $\alpha$ are the subject of this section, and it is seen that all such subgroups are closely associated with subgroups tidy for~$\alpha$. Several characterizations of the nub are derived in addition, and the next few paragraphs explain the significance of these alternative characterizations. 

The terms `tidy above' and `tidy below' refer to a procedure for determining the scale of the automorphism~$\alpha$ and finding a minimizing subgroup. This procedure starts with a given compact, open subgroup, $V$, and modifies it in several steps to produce a minimizing subgroup, see \cite{Wi:structure,wi:further}. The aim at each step is to reduce the index $[\alpha(V) : \alpha(V)\cap V]$.

In the first step, $n$ is chosen as in Lemma~\ref{lem:recalltidy} to find $V'\dfn \bigcap_{k=0}^n \alpha^k(V)$ that satisfies $\text{\bf TA}(\alpha)$. This step reduces $\alpha(V)$ relative to $\alpha(V)\cap V$ by cutting down $V$, thus making $V$ `tidy above'. 

The next step defines the group $L$ to be the closure of 
$$
\mathscr{L} \dfn \left\{ x\in G \mid \alpha^k(x) \in V' \hbox{ for all but finitely many }k\right\}
$$
and shows that $L$ is compact. 

The final step combines $L$ with $V'$ to obtain a compact, open subgroup that satisfies $\text{\bf TB}(\alpha)$ as well. In the papers \cite{Wi:structure} and \cite{wi:further} this is done in two different ways that are described in Remark~\ref{rem:TBfirst}. This step increases $\alpha(V)\cap V$ relative to $\alpha(V)$ by adding elements to the intersection, thus making $V$ `tidy below'. 

As is apparent from its definition, $L$ depends on the subgroup $V'$ produced in the first step. Tidiness below cannot therefore be achieved by this procedure until after tidiness above. The main aim of this section is to show that $\nub{\alpha}$, the intersection of all subgroups tidy for $\alpha$, may be used in place of $L$ irrespective of the subgroup $V$ that is given. 

\subsection{Tidiness and the nub}
\label{sec:tidy&nub}

The next proposition is essentially \cite[Lemma~3.31(3)]{BaumW_Cont} but, unlike the proof there, does not rely on non-triviality of $\con{\alpha}$. 

\begin{proposition}
\label{prop:nubimpliesT2}
Let $\alpha$ be an automorphism of $G$. Then a compact, open subgroup that contains $\nub{\alpha}$ and satisfies $\text{\bf TA}(\alpha)$ is tidy for $\alpha$. 
\end{proposition}
\begin{proof}
Suppose that $W$ is a compact open subgroup of $G$ with $\nub{\alpha}\leq W$ and $W = W_+W_-$. Denote $W_+\cap W_- = \bigcap_{n\in {\mathbb Z}} \alpha^n(W)$ by $W_0$. Applying the tidying procedure to $W$ it suffices, in order to see that $W$ is tidy, to show that the subgroup
$$
\mathscr{L} \dfn \left\{ x\in G \mid \alpha^k(x) \in W \hbox{ for all but finitely many }k\right\}
$$
is contained in $W$. 

Let $x\in \mathscr{L}$ and observe that there is $k\in \mathbb{Z}$ such that $\alpha^k(x)\in W_+$. Without loss of generality, suppose that $x\in W_+$. Since $W$ is a neighbourhood of $\nub{\alpha}$, there is a subgroup $V\leq W$ that is tidy for $\alpha$. Next, since $W_0V$ is a neighbourhood of $W_0$ and $\bigcap_{n\geq0} \alpha^{-n}(W_+) = W_0$ is an intersection of compact sets, there is an integer $m$ such that $\alpha^{-m}(W_+)\subset W_0V$. Hence, replacing $x$ by $\alpha^{-m}(x)$ for sufficiently large $m$, it may be supposed that $x \in W_0V$. Writing $x = x_0v$, the orbit $\{\alpha^n(v)\}_{n\in {\mathbb Z}}$ is bounded because $\alpha^n(v) = \alpha^n(x_0)^{-1}\alpha^n(x)$ for each $n$. Since $v$ belongs $V$ which is tidy for $\alpha$, it follows by Lemma~\ref{lem:criterion}(\ref{lem:criterionii}) that $v\in V_0  \dfn \bigcap_{n\in {\mathbb Z}} \alpha^n(V)\leq W_0$. Therefore $x\in W_0V_0 = W_0$ and $\alpha^n(x)\in W$ for all $n$ as required. 
\end{proof}

Since the nub is $\alpha$-stable, it follows from Lemma~\ref{lem:recalltidy}(\ref{lem:recalltidyi}) that 
\begin{corollary}
\label{cor:nubT2auto}
Let $V$ be a compact, open subgroup of $G$ and suppose that $V$ contains $\nub{\alpha}$. Then there is $n\in \mathbb{N}$ such that $\bigcap_{k=0}^n \alpha^k(V)$ is tidy for $\alpha$. \endproof
\end{corollary}

A close relationship between $\nub{\alpha}$ and condition $\text{\bf TB}(\alpha)$ also follows. 
\begin{corollary}
\label{cor:nubimpliesT2}
The compact, open subgroup, $V\leq G$ is tidy below for $\alpha\in \Aut(G)$ if and only if $V$ contains $\nub{\alpha}$. 
\end{corollary}
\begin{proof}
It must be shown that $V_{++}$ is closed if and only if $\nub{\alpha}\leq V$. Choose $n\in {\mathbb N}$ such that $W := \bigcap_{k=0}^n \alpha^k(V)$ is tidy above for $\alpha$. Then $W_+ = V_+$ and $W_{++} = V_{++}$. Suppose that $\nub{\alpha}\leq V$. Then $\nub{\alpha}\leq W$ and $W_{++} = V_{++}$ is closed by Proposition~\ref{prop:nubimpliesT2}. Conversely, if $V_{++}$ is closed, then so is $W_{++}$ and $W$ is tidy for $\alpha$. Then $\nub{\alpha}\leq W \leq V$.   
\end{proof}

\subsection{Tidiness and stability}
Tidy subgroups are minimizing for $\alpha$ and, in that sense, come as close as is possible for a  compact, open subgroup to be stable under~$\alpha$. It is seen next that stable compact subgroups are closely associated with tidy ones.
   
\begin{proposition}
\label{prop:nub}
Let $\alpha\in\Aut(G)$ and suppose that $N$ is a compact, $\alpha$-stable subgroup of~$G$ with $\nub{\alpha}\leq N$. Then:
\begin{enumerate}
\item for any open $\open{U}\supset N$ there is an $\alpha$-tidy subgroup $V$ with $N\leq V \subset \open{U}$;
\label{lem:nubi}
\item for any relatively open, $\alpha$-stable subgroup, $J$, of $N$ there is an $\alpha$-tidy subgroup,~$W$, such that $J = W\cap N$. 
\label{lem:nubii}
\end{enumerate}
\end{proposition}
\begin{proof}
(\ref{lem:nubi}) Since $\open{U}$ is an open neighbourhood of $N$, there is an open subgroup, $W$, with $N\leq W\subset \open{U}$. Then $\nub{\alpha}\leq W$ and so, by Corollary~\ref{cor:nubT2auto}, there is an integer $n$ such that $\bigcap_{k=0}^n \alpha^k(W) =: V$ is $\alpha$-tidy. It is clear that $V\subset \open{U}$ and that $N\leq V$ follows from the fact that $N$ is $\alpha$-stable. 

(\ref{lem:nubii}) 
(The claim would follow immediately from (\ref{lem:nubi}) if $\nub{\alpha}$ were contained in~$J$. That is not being assumed however.) Since $J$ is open in $N$, there is an open subgroup,  $V_1$ say, of $G$ such that $V_1J \cap N = J$. The smaller open subgroup 
$$
V_2 := \bigcap\left\{ xV_1x^{-1} \mid x\in N \right\}
$$ 
is normalized by $N$, so that the products $V_2 N$ and $V_2J$ are subgroups of~$G$. 
The even smaller subgroup 
$$
V_3 := V_2\cap \alpha^{-1}(V_2)
$$ 
is normalized by $N$, because $N$ is $\alpha$-stable, and satisfies $(V_3J)\alpha(V_3J)\subset V_2J$, because $J$ is $\alpha$-stable. Hence $V_3J$ is an open subgroup of $G$ and satisfies
\begin{equation}
\label{eq:V3}
(V_3J)\alpha(V_3J) \cap N \leq V_2J \cap N = J.
\end{equation}

By part (\ref{lem:nubi}), there is an $\alpha$-tidy subgroup, $V$, between $N$ and $V_3N$. Then 
\begin{equation}
\label{eq:V'defn}
W := V\cap V_3J
\end{equation}
is a compact, open subgroup of $G$ and $W\cap N = J$. The desired conclusion will be reached once it is shown that  $W$ is tidy. 

To this end, note first that $V = WN$ because each $v\in V$ is a product $v = xy$ with $x\in V_3$ and $y\in N$ where, moreover, $x = vy^{-1}$ belongs to $V_3\cap V$, so that in fact $x\in W$. Next, consider  
$$
\alpha(W)\cap V = \alpha(W)\cap WN.
$$
If $x$ belongs to this intersection, then $x = wz$ where $w\in W$ and $z\in N$. Then $z = w^{-1}x\in W\alpha(W)\cap N =J$, by (\ref{eq:V'defn}) and (\ref{eq:V3}). Hence $\alpha(W)\cap V \leq W$ and the map
$$
 \alpha(W)/\left(\alpha(W)\cap  W\right) \to \alpha(V)/\left(\alpha(V)\cap V\right) : x\left(\alpha(W)\cap W\right) \mapsto x\left(\alpha(V)\cap V\right)
 $$
  is injective. It follows that
  $$
  \left[\alpha(W) : \alpha(W)\cap W\right]  \leq \left[\alpha(V) : \alpha(V)\cap V\right]. 
  $$
  However the index on the right is the minimum possible because $V$ is tidy for $\alpha$ and it follows that $W$ is minimizing and therefore is tidy too. 
\end{proof}

The special case when $N = \nub{\alpha}$ yields \cite[Lemma~3.31(2)]{BaumW_Cont}, but without using density of $\con{\alpha}$ in the proof.
\begin{corollary}
\label{cor:nub}
The nub of ${\alpha}$ has no proper, relatively open $\alpha$-stable subgroups. \endproof
\end{corollary}

The first part of the next result follows from \cite[Lemma~3.19]{BaumW_Cont} but a short proof is included here. 
\begin{proposition}
\label{prop:tidymeetsstable}
Let $\alpha\in \Aut(G)$. Then the following hold for any $V\in {\mathcal B}(G)$. 
\begin{enumerate}
\item If $V$ is tidy for $\alpha$, then $V\cap K$ is $\alpha$-stable for any compact $\alpha$-stable subgroup $K$ of $G$. \label{lem:tidymeetsstablei}
\item If $V\cap K$ is $\alpha$-stable for any compact $\alpha$-stable subgroup $K$ of $G$, then $V$ is tidy below for $\alpha$ but need not be tidy above. \label{lem:tidymeetsstableii}
\end{enumerate}
\end{proposition}
\begin{proof}
(\ref{lem:tidymeetsstablei}) Let $x\in V\cap K$. Then $\{\alpha^n(x)\}$ is bounded because it is contained in $K$. Hence $\{\alpha^n(x)\}$ is contained in $V$ by Lemma~\ref{lem:criterion}.

(\ref{lem:tidymeetsstableii}) If $V\cap K$ is $\alpha$-stable for any compact $\alpha$-stable $K$, then in particular $V\cap \nub{\alpha}$ is $\alpha$-stable. Since $V\cap \nub{\alpha}$ is an open subgroup of $\nub{\alpha}$, Corollary~\ref{cor:nub} implies that $\nub{\alpha}\leq V$, so that $V$ is tidy below. 

The pair $(F^{[\mathbb{Z}]},\sigma)$, where $F$ is a finite group and $F^{[\mathbb{Z}]}$ has the discrete topology, has the property that only the trivial subgroup is compact and either $\sigma$-stable or $\sigma$-tidy. Thus every open subgroup satisfies that its intersection with any $\sigma$-stable subgroup is $\sigma$-stable, and yet no non-trivial compact, open subgroup is tidy. 
\end{proof}

%REMARK: Does this lemma hold for $V$ tidy below? $V$ is tidy (below) if and only if $V\cap K$ is $\alpha$-stable for every $\alpha$-stable $K$. ANSWER: No. If $G$ is discrete then any subgroup is tidy below but need not have $\alpha$-stable intersection with a finite $\alpha$-stable group. 

If $K$ is compact, $\alpha$-stable and has no relatively open $\alpha$-stable subgroups, then Proposition~\ref{prop:tidymeetsstable}(\ref{lem:tidymeetsstablei}) implies that $K$ is contained in every compact, open  subgroup of~$G$ that is tidy for $\alpha$, thus yielding the following characterization of $\nub{\alpha}$.
\begin{corollary}
\label{cor:nubmax}
The nub of ${\alpha}$ is the largest compact, $\alpha$-stable subgroup of $G$ having no relatively open, proper, $\alpha$-stable subgroups. 
\endproof
\end{corollary}

\subsection{Automorphisms acting ergodically}
When $G$ is compact, an open subgroup is tidy for $\alpha$ in $\Aut(G)$ if and only if it is $\alpha$-stable. Hence  $\nub{\alpha}$ is equal to the intersection of all open $\alpha$-stable subgroups of $G$. This intersection is described more precisely in \cite[Theorem~1.4]{KSchmidt}\footnote{This reference treats only the metrizable case.} and \cite[Theorem~2.6]{WJaw2}, which apply to general compact groups and not just totally disconnected ones. Specializing to the totally disconnected case, these papers establish the existence of an ordinal $\tau$ and a decreasing family $(G_\rho)_{\rho\in [0,\tau]}$ of closed, $\alpha$-stable subgroups of $G$ such that: $G_0=G$; $G_{\rho}/G_{\rho+1}$ is finite for every $\rho\in [0,\tau)$; and $\alpha$ acts ergodically on $G_\tau$, which is the largest closed, $\alpha$-stable subgroup of $G$ on which $\alpha$ acts ergodically. 

The subgroup $G_\tau$ in \cite[Theorem~2.6]{WJaw2}, which is denoted by $G_{erg}$ in that paper, is equal to $\nub{\alpha}$. Since, as shown in~\cite{Aoki}, any group on which an automorphism acts ergodically must be compact, Corollary~\ref{cor:nubmax} and \cite{WJaw2} imply the following. 

\begin{proposition}
\label{prop:nubergodic} 
Let $G$ be a totally disconnected, locally compact group and let $\alpha\in \Aut(G)$. Then the action of $\alpha$ on $\nub{\alpha}$ is ergodic and $\nub{\alpha}$ is the largest closed, $\alpha$-stable subgroup of $G$ on which $\alpha$ acts ergodically.
\endproof
\end{proposition}

Recall that the the word `ergodic' in the last proposition could be replaced by `topologically transitive'. The next result is a restatement of~\cite[Corollary~2.7(ii)]{WJaw2}. 
\begin{proposition}
\label{prop:ErgQuot}
Let $G$ be a compact group and $\alpha\in \Aut(G)$. Suppose that $N \triangleleft G$ is a closed, normal $\alpha$-stable subgroup. Then $\nub{\alpha|^{G/N}} = {\mathfrak q}_N(\nub{\alpha})$. 
\endproof
\end{proposition}
%\begin{proof}
%The group ${\mathfrak q}_N(\nub{\alpha})$ has  no proper open $\alpha$-stable subgroups because it is isomorphic to $\nub{\alpha}/(\nub{\alpha}\cap N)$. Hence ${\mathfrak q}_N(\nub{\alpha})\leq \nub{\alpha|^{G/N}}$ by Corollary~\ref{cor:nubmax}. 

%For the reverse inclusion, it suffices to show that 
%\begin{equation}
%\label{eq:MinStabQuot}
%{\mathfrak q}_N^{-1}\left(\nub{\alpha|^{G/N}}\right) \leq \nub{\alpha}N.
%\end{equation}
%To this end, let $\{V_i\}_{i\in{\mathcal I}}$ be the set of open $\alpha$-stable subgroups of $G$. Then ${\mathfrak q}_N(V_iN)$ is an open $\alpha|^{G/N}$-stable subgroup of $G/N$ for each $i\in {\mathcal I}$.  Hence ${\mathfrak q}_N^{-1}(\nub{\alpha|^{G/N}})$ is contained in $V_iN$, and it follows that  
%$$
%C_x^{(i)} := \left\{ (v_i,n_i)\in V_i\times N \mid v_in_i = x\right\} \ne \emptyset
%$$
%for each $x\in {\mathfrak q}_N^{-1}\left(\nub{\alpha|^{G/N}}\right)$ and $i\in {\mathcal I}$.
 % Since $C_x^{(i)} $ is a closed subset of the compact set $G\times N$ and $\{V_i\}_{i\in {\mathcal I}}$ is a directed set of subgroups, 
%  $$
%  \bigcap_{i\in {\mathcal I}} C_x^{(i)} \ne \emptyset\text{ for each }x\in {\mathfrak q}_N^{-1}\left(\nub{\alpha|^{G/N}}\right).
%  $$
%Hence $\bigcap_{i\in {\mathcal I}} C_x^{(i)} \leq \bigcap_{i\in {\mathcal I}} V_i\times N = \nub{\alpha}\times N$, thus establishing  (\ref{eq:MinStabQuot}).
%\end{proof}
\begin{remark} {\it(a)\/} Example~6.4 from \cite{wi:further} gives a non-compact group $G$ and $\alpha$ in $\Aut(G)$ such that $\nub{\alpha}=\triv$ and yet $G$ has a (discrete) normal subgroup $N$ such that $G/N$ is compact and $\nub{\alpha|^{G/N}}  = G/N$. Hence Proposition~\ref{prop:ErgQuot} may therefore fail if $G$ is not assumed to be compact.\\
{\it(b)\/} While it is clear that ${\mathfrak q}_N(\bcg{\alpha}) \leq \bcg{\alpha|^{G/N}}$, Example~\ref{ex:con_2-ways_not_dense} shows that equality may fail. 
\end{remark}

The proofs of the next claims are straightforward.
\begin{proposition}
\label{prop:normal_subs}
Let $G$ be a totally disconnected, compact group and $\alpha\in \Aut(G)$. Then $\nub{\alpha}$, $\con{\alpha}$ and  $\bcg{\alpha}$ are normal subgroups of $G$.
\endproof
\end{proposition}
%\begin{proof}
%An open subgroup of $G$ is tidy for $\alpha$ if and only if it is $\alpha$-stable. Hence $\nub{\alpha}$ is the intersection of all open $\alpha$-stable subgroups. If $V$ is any such subgroup, then $\bigcap_{x\in G} xVx^{-1}$ is open, $\alpha$-stable and normal in $G$. It follows that $\nub{\alpha}$ is the intersection of all open, normal, $\alpha$-stable subgroups and is therefore itself normal.

%Let $x\in \con{\alpha}$ and $y\in G$. Then $\alpha^n(yxy^{-1}) = \alpha^n(y)\alpha^n(x)\alpha^n(y^{-1})\to \ident$ because $\{\alpha^n(y)\}_{n\geq0}$ is bounded. Hence $\con{\alpha}$ is normal in $G$. That $\bcg{\alpha}$ is also normal follows because $\bcg{\alpha} = \con{\alpha}\cap \con{\alpha^{-1}}$. 
%\end{proof}

\subsection{Characterizations of $\nub{\alpha}$ and finding tidy subgroups}  The results of this section and elsewhere yield the following characterizations of $\nub{\alpha}$.  
  
\begin{theorem}
\label{thm:charnub}
Let $G$ be a totally disconnected, locally compact group and $\alpha$ be an automorphism of $G$. Then $\nub{\alpha}$ is the:
\begin{enumerate}
\item intersection of all $\alpha$-tidy subgroups of $G$; \label{thm:charnub1}
\item closure of $\bcon{\alpha}$; \label{thm:charnub2}
\item intersection $\bigcap\left\{ \overline{\rbcon{\alpha}{V}}\mid {V\in \mathcal{B}(G)}\right\}$;  \label{thm:charnub3}
\item largest compact, $\alpha$-stable subgroup of $G$ having no relatively open $\alpha$-stable subgroups; and \label{thm:charnub4}
\item largest compact, $\alpha$-stable subgroup of $G$ on which $\alpha$ acts ergodically. \label{thm:charnub5}
\end{enumerate}
 \end{theorem}
\begin{proof}
The statement in~(\ref{thm:charnub1}) is the definition of $\nub{\alpha}$, while (\ref{thm:charnub4}) and (\ref{thm:charnub5}) are respectively Corollary~\ref{cor:nubmax} and Proposition~\ref{prop:nubergodic} above. 

Since, as seen in Equation~(\ref{eq:bcon_and_rbcon}),
$$
\overline{\bcon{\alpha}}\leq \bigcap\left\{ \overline{\rbcon{\alpha}{V}}\mid {V\in \mathcal{B}(G)}\right\},
$$
to prove (\ref{thm:charnub2}) and (\ref{thm:charnub3}), it suffices to show that 
$$
\nub{\alpha}\leq \overline{\bcon{\alpha}}\hbox{ and }\bigcap\left\{ \overline{\rbcon{\alpha}{V}}\mid {V\in \mathcal{B}(G)}\right\}\leq \nub{\alpha}.
$$
The first containment holds because, by  Corollary~\ref{cor:nubmax}, $\nub{\alpha}$ has no compact, open subgroups, and so, by Proposition~\ref{prop:con_dense}, its contraction group is dense. 

For the second, following Equation~(4) in \cite{Wi:SimulTriang}, denote 
$$
\bigcap\left\{ \overline{\rbcon{\alpha}{V}}\mid {V\in \mathcal{B}(G)}\right\} =: K.
$$
Then $K$ is compact by \cite[Lemma~2.2]{Wi:SimulTriang} and is $\alpha$-stable because $\mathcal{B}(G)$ is invariant under conjugation. Therefore $K$ is contained in the maximal compact $\alpha$-stable subgroup of $G$ which, by Corollary~\ref{cor:nubmax}, is equal to $\nub{\alpha}$.
\end{proof}

\begin{remark}
\label{rem:TBfirst}
The characterizations of $\nub{\alpha}$ given in Theorem~\ref{thm:charnub}(\ref{thm:charnub2})--(\ref{thm:charnub5}) admit the following `algorithm' for finding subgroups tidy for $\alpha$, that is an alternative to those given in \cite{Wi:structure,wi:further,Wi:SimulTriang}. Let $V$ be any compact, open subgroup of $G$ and suppose that $\nub{\alpha}$ has been identified by one of these characterizations. Then $\nub{\alpha}$ may be embedded in a compact, open subgroup of $G$ obtained from $V$ by either of the two following methods. 

The first method defines  $V'' \dfn \bigcap_{x\in \nub{\alpha}} xVx^{-1}$, which is normalized by $\nub{\alpha}$, and is open because there are only finitely many distinct conjugates $xVx^{-1}$ as $x$ ranges over~$\nub{\alpha}$. Hence $V''\nub{\alpha}$ is an open subgroup of $G$ containing $\nub{\alpha}$. The second method is given in \cite{wi:further}. Define $V''' \dfn \left\{v\in V \mid \nub{\alpha}v\subset V\nub{\alpha}\right\}$, 
which is an open subgroup of $G$ and $V'''\nub{\alpha} = \nub{\alpha}V'''$, see~\cite[Lemma~3.3]{wi:further}. Hence $V'''\nub{\alpha}$ is an open subgroup of $G$ containing $\nub{\alpha}$.

Either method produces a compact, open subgroup, $U$, of $G$ containing $\nub{\alpha}$, which therefore satisfies $\text{\bf TB}(\alpha)$, by Corollary~\ref{cor:nubimpliesT2}. There is then a positive integer~$n$ such that $\bigcap_{k=0}^n \alpha^k(U)$ is tidy for $\alpha$, by Corollary~\ref{cor:nubT2auto}. 
\end{remark}

\section{Automorphisms of compact groups}
\label{sec:contraction} 

As seen in the previous section, the intersection of all subgroups tidy for $\alpha$ is the compact subgroup $\nub{\alpha}$ and the restriction of $\alpha$ to $\nub{\alpha}$ is topologically transitive. The detailed study of topologically transitive pairs $(G,\alpha)$ begun in this section thus contributes to a more general study of automorphisms of locally compact groups. The approach taken is the same as that in ergodic theory, namely, to express $(G,\alpha)$ as an inverse limit of pairs having a finiteness property called \emph{expansiveness} in ergodic theory and here called \emph{finite depth}. 

The description of pairs $(G,\alpha)$ as inverse limits is similar to, and is in fact derived from, the theorem that each compact totally disconnected group is an inverse limit of finite groups. Furthering this similarity, each pair  $(G,\alpha)$ with finite depth has a \emph{depth}, which is a positive integer analogous to the order of a finite group. This similarity is behind a more complete description of such pairs in Section~\ref{sec:Jordan-Holder}, where it is shown that a version of the Jordan-H\"older Theorem holds. 

\subsection{The finite depth condition}
\label{sec:finite_depth}

\begin{definition}
\label{defn:separating}
The pair $(G,\alpha)$ of a compact group and automorphism has \emph{finite depth} if there is an open subgroup $V\leq G$ such that $\bigcap_{k\in{\ZZ}} \alpha^k(V) = \triv$. 
\end{definition}

\begin{remark}
\label{rem:separating}
{\it (a)\/} 
In topological dynamics, an automorphism satisfying the finite depth condition is called \emph{expansive}. That term is not adopted here however because it has misleading connotations. An automorphism of a compact group cannot `expand' sets: if an open set $\open{U}\subset G$ satisfies $\alpha(\open{U}) \supset \open{U}$, then $\alpha(\open{U}) = \open{U}$. The term `expansive' is also inconsistent \cite{BaumW_Cont}, where it is said that, if $V = V_+V_-$ is tidy for $\alpha$, then $\alpha$ `expands' $V_+$ and `shrinks' $V_-$ and that $s(\alpha)$ is the `expansion factor'. 
\\
 {\it (b)\/} If $V$ is an open subgroup of the compact group $G$ satisfying $\bigcap_{k\in{\ZZ}} \alpha^k(V) = \triv$, then $\bigcap_{x\in G} xVx^{-1}$ is a normal, open subgroup of $G$ satisfying the same condition. Hence the subgroup~$V$ in Definition~\ref{defn:separating} may be taken to be normal. This subgroup may also be taken to be tidy above for $\alpha$ by replacing it by $\bigcap_{k=0}^n \alpha^k(V)$ if necessary, see Lemma~\ref{lem:recalltidy}. These additional conditions will usually be assumed to be satisfied.
\end{remark}

The following description of compact pairs $(G,\alpha)$ may also be derived from \cite[Theorem~5.3]{WJaw2} and \cite[Theorem~3.6]{KitchSchmidt}\footnote{In the case of metrizable groups.}.
\begin{proposition}
\label{prop:Proj_Limit}
Let $G$ be a compact, totally disconnected group and $\alpha\in \Aut(G)$. Then there is an inverse system, $\left\{(G_i,\alpha_i), \varphi_{ij}, {\mathcal I}\right\}$ with each pair $(G_i,\alpha_i)$ having finite depth such that
\begin{equation}
\label{eq:Proj_Limit}
(G,\alpha) \cong \varprojlim (G_i,\alpha_i).
\end{equation}
\end{proposition}
\begin{proof}
Since $G$ is compact, it has a base, $\left\{V_i\right\}_{i\in \mathcal I}$, of neighbourhoods of $\ident_G$ consisting of open, normal subgroups. For each $i\in{\mathcal I}$, define  $N_i = \bigcap_{k\in{\ZZ}} \alpha^k(V_i)$, $G_i = G/N_i$ and $\alpha_i = \alpha|^{G_i}$. Then $G_i$ is a compact group, $\alpha_i$ is a well-defined automorphism, and $(G_i,\alpha_i)$ has finite depth. Further define, for $i,j$ with $V_j \leq V_i$,  $\varphi_{i,j} : G_j\to G_i$ to be the quotient by $N_i/N_j$.  Then $\left\{G_i,\varphi_{i,j},{\mathcal I}\right\}$ is a projective system of compact groups and $\varphi_{i,j}\circ\alpha_j = \alpha_i\circ \varphi_{i,j}$. It is routine to check that $(G,\alpha)$ is the claimed inverse limit.
\end{proof}

The next result is a special case of \cite[Proposition~3.5]{KSchmidt}. It allows the conditions `$\nub{\alpha} = G$' and `$G$ has no proper open $\alpha$-stable subgroups' to be used interchangeably. The proof given here is based on Lemma~\ref{lem:recalltidy}. 
\begin{lemma}
\label{lem:alphastable_subgroup}
For any compact pair  $(G,\alpha)$ with finite depth, $\nub{\alpha}$ is an open subgroup of~$G$.
\end{lemma}
\begin{proof}
Choose an open subgroup, $V$ of $G$, such that $\bigcap_{k\in{\mathbb Z}} \alpha^k(V) = \triv$ and $V= V_+V_-$, see Remark~\ref{rem:separating}. Let $W$ be an open $\alpha$-stable subgroup of $G$. It will be shown that both $V_+$ and $V_-$ are subgroups of $W$, whence $\nub{\alpha}\geq V$. 

Consider first $x\in V_+$. Since $\alpha^{-n}(x)\in \bigcap_{k=-n}^\infty \alpha^k(V)$ for each $n$,  $\alpha^{-n}(x)\to\ident$ as $n\to\infty$. Hence $\alpha^{-n}(x)$ belongs to $W$ for $n$ sufficiently large. 
Since $W$ is $\alpha$-stable, $x\in W$ and so $V_+\leq W$. That $V_-\leq W$ may be shown in a similar way.
\end{proof}

The next proposition is the basis for making the notion of `depth' precise. 
\begin{proposition}
\label{prop:Possame}
Suppose that $(G,\alpha)$ is infinite but has finite depth. Let $V$ be an open subgroup of $G$ with $\bigcap_{k\in{\ZZ}} \alpha^k(V) = \triv$ and $V = V_+V_-$. Then $[\alpha(V_+):V_+]$ is strictly greater than~$1$ and is independent of the choice of $V$ with these properties. 
\end{proposition}
\begin{proof}
Since $G$ is compact, $m(\alpha(V)) = m(V)$, where $m$ is Haar measure, and so
$$
[\alpha(V_+) : V_+ ]  = [V_- : \alpha(V_-)] . 
$$
Hence, if $[\alpha(V_+) : V_+] = 1$, then  $V_+$ and $V_-$ are both $\alpha$-stable and, consequently, so is $V$. This can only occur if $V$ is the trivial subgroup and $G$ is finite, a contradiction.

Let $V'$ be a second open subgroup, with $\bigcap_{k\in{\ZZ}} \alpha^k(V') = \triv$ and $V' = V'_+V'_-$. Then there is an $n\in{\NN}$ such that $\bigcap_{k=-n}^n \alpha^k(V') \leq V$. In particular, $\alpha^{-n}(V'_+)\leq V_+$. Consideration of the inclusions $\alpha(V_+) \geq V_+ \geq \alpha^{-n}(V'_+) \geq \alpha^{-n-1}(V'_+)$ yields that
\begin{eqnarray*}
[\alpha(V_+) : \alpha^{-n-1}(V'_+)] &=& [\alpha(V_+) : V_+][ V_+ : \alpha^{-n-1}(V'_+)] \\
&=& [\alpha(V_+) : \alpha^{-n}(V'_+)][ \alpha^{-n}(V'_+) : \alpha^{-n-1}(V'_+)]. 
\end{eqnarray*}
Since $\alpha$ is an automorphism, $[ V_+ : \alpha^{-n-1}(V'_+)]  =  [\alpha(V_+) : \alpha^{-n}(V'_+)]$ and so 
$$
[\alpha(V_+) : V_+] = [ \alpha^{-n}(V'_+) : \alpha^{-n-1}(V'_+)] = [\alpha(V'_+) : V'_+].
$$
\end{proof}

\begin{definition}
\label{defn:depth}
The \emph{depth} of $(G,\alpha)$ with finite depth is the index $[\alpha(V_+): V_+]$, where $V$ is any open subgroup of $G$ with $\bigcap_{k\in{\ZZ}} \alpha^k(V) = \triv$ and $V = V_+V_-$. 
\end{definition}

The entropy of the action of $\alpha$ on $G$ is the logarithm of the depth, \cite[Theorem~2]{Kitchens}. 
The next result, a consequence of Lemma~\ref{lem:recalltidy}, will be used in Section~\ref{sec:Jordan-Holder} to describe the algebraic structure of pairs $(G,\alpha)$ with finite depth. 
\begin{lemma}
\label{lem:capfinite}
Suppose that $(G,\alpha)$ has finite depth and let $V$ be an open, normal subgroup of $G$ such that $\bigcap_{k\in\ZZ} \alpha^k(V) = \triv$. Then $\alpha^i(V_+)\cap \alpha^j(V_-)$ is a finite normal subgroup of $G$ for every $i,j\in\ZZ$. 
\end{lemma}
\begin{proof}
Since $\alpha^{-1}(V_+)\leq V_+$ and $\alpha(V_-)\leq V_-$, and since $V_+\cap V_- = \triv$, it suffices to consider the case when $j=0$ and $i>0$. For this, note that the quotient map $\alpha^i(V_+) \to \alpha^i(V_+)/V_+$ is injective on $\alpha^i(V_+)\cap V_-$ because $V_+\cap V_- = \triv$. The claim then follows because $\alpha^i(V_+)/V_+$ is finite, see Lemma~\ref{lem:recalltidy}.
\end{proof}

\subsection{Density of the contraction group}
\label{sec:condensity}

It is shown in \cite[Theorem~3.26]{BaumW_Cont} that, if $G$ is a metrizable, locally compact group and $\alpha\in\Aut{G}$, then $\con{\alpha}$ is dense in $\ret{\alpha}$, and it was shown in \cite{WJaw} that the metrizability restriction could be dispensed with. An alternative proof of the special case when $G$ is compact is given below. Rather than the compactness argument given in \cite{BaumW_Cont}, the proof uses a more constructive algebraic argument to treat the finite depth case followed by `approximation by finite depth' using Proposition~\ref{prop:Proj_Limit}. 

\begin{proposition}
\label{prop:prodeq}
Suppose that the pair $(G,\alpha)$ has finite depth and is topologically transitive. Let $V$ be an open, normal subgroup of $G$ such that $\bigcap_{k\in\mathbb{Z}} \alpha^k(V) = \triv$ and $V = V_+V_-$. Then 
\begin{enumerate}
\item there is $k\in {\mathbb N}$ such that $\alpha^k(V_+)V_- = G$; and
\label{cor:prodeqi}
\item $\bigcup_{n\in{\NN}} \alpha^n(V_+)\cap \alpha^{-n}(V_-)$ is dense in $G$.
\label{cor:prodeqii}
\end{enumerate}
\end{proposition}

\begin{proof}
(\ref{cor:prodeqi}) The hypothesis that $V$ is normal implies that $V_+$ and $V_-$ are normal subgroups of $G$ as well. Hence $\left\{ \alpha^k(V_+)V_-\right\}_{k\in\mathbb{N}}$ is an increasing sequence of subgroups of $G$ that are open because they contain $V$. Since open subgroups have finite index in $G$, this increasing sequence is eventually constant and equal to $\alpha^k(V_+)V_-$ for some $k$. Then 
$$
\alpha\left(\alpha^k(V_+)V_-\right) \subset \alpha^{k+1}(V_+)V_- = \alpha^k(V_+)V_-
$$
and so, since $G$ is compact and $\alpha$ therefore preserves Haar measure, $\alpha^k(V_+)V_-$ is $\alpha$-stable. The hypothesis that $G$ has no proper open, $\alpha$-stable subgroups then implies that $\alpha^k(V_+)V_- = G$. 

(\ref{cor:prodeqii}) It will first be shown that $\bigcup_{n\in \mathbb{N}} \left( \alpha^n(V_+)\cap V_-\right)$ is dense in $V_-$. For this, consider $x\in V_-$. From part~(\ref{cor:prodeqi}), we have that $x\in \alpha^{k+l}(V_+)\alpha^l(V_-)$ for every $l\in \mathbb{Z}$ whence, for each $l>0$, there is $y\in \alpha^l(V_-)$ such that $xy\in \alpha^{k+l}(V_+)\cap V_-$. Since $\left\{\alpha^l(V_-)\right\}_{l\geq 0}$ is a sequence of subgroups that decreases to the trivial subgroup, it follows that $\bigcup_{n\in \mathbb{N}} \left( \alpha^n(V_+)\cap V_-\right)$ is dense in $V_-$ as claimed. It may be shown by a similar argument that $\bigcup_{n\in \mathbb{N}} \left( \alpha^k(V_+)\cap \alpha^{-n}(V_-)\right)$ is dense in $\alpha^k(V_+)$. This suffices to complete the proof because $G = \alpha^k(V_+)V_-$. 
\end{proof}

Since $\alpha^n(V_+)\cap \alpha^{-n}(V_-)$ is contained in $\bcg{\alpha}$, we have the following. 
\begin{corollary}
\label{cor:bcon_dense}
Let $(G,\alpha)$ be a compact, topologically transitive pair with finite depth. Then $\bcg{\alpha}$ is dense in~$G$.
\endproof
\end{corollary}

The contraction group is dense even when $(G,\alpha)$ does not have finite depth. 
\begin{proposition}
\label{prop:con_dense}
Let $(G,\alpha)$ be a compact, topologically transitive pair. Then 
\begin{enumerate}
\item $\con{\alpha}$ is dense in $G$, and \label{prop:con_densei}
\item $\overline{\bcg{\alpha}} \supset [G,G]$.
\label{prop:con_denseii}
\end{enumerate}
\end{proposition}
\begin{proof}
(\ref{prop:con_densei}) By Proposition~\ref{prop:Proj_Limit}, $G$ is the inverse limit of pairs $(G_i,\alpha_i)$ with finite depth where each $G_i$ is a quotient, $G/N_i$. Then $G_i$ inherits from $G$ the property of having no proper open, $\alpha$-stable subgroups and so, by Corollary~\ref{cor:bcon_dense}, $\con{\alpha_i}$ is dense in $G_i$. Hence the contraction subgroup of $\alpha$ modulo $N_i$ is dense for each $i$ and it follows, by \cite[Theorem~1]{WJaw}, that $\con{\alpha}$ is dense in $G$. 

(\ref{prop:con_denseii}) Part (\ref{prop:con_densei}) implies that $\con{\alpha}$ and $\con{\alpha^{-1}}$ are both dense in $G$. Since $\con{\alpha}$ and $\con{\alpha^{-1}}$ are  normal subgroups of $G$,
$$
[\con{\alpha},\con{\alpha^{-1}}] \subset \con{\alpha}\cap \con{\alpha^{-1}} = \bcg{\alpha}
$$
and the claim follows.
\end{proof}

It might be thought that $\bcg{\alpha}$ would also be dense in $G$ by the reasoning used to establish Proposition~\ref{prop:con_dense}(\ref{prop:con_densei}). However, the next example, which is the same as \cite[Examples~5.6]{KSchmidt}, shows that that is not always the case. 

\begin{example}
\label{ex:con_2-ways_not_dense}
 Let $C_p$ be the cyclic group of order~$p$ and let $\sigma$ be the shift on $C_p^{\mathbb{Z}}$, see Subsection~\ref{sec:examples}. The map $\varphi : C_p^{\mathbb{Z}} \to C_p^{\mathbb{Z}}$ defined by
 $$
 \varphi(f)(n) = f(n) -_p f(n + 1)
 $$ 
 is a surjective homomorphism which commutes with $\sigma$ and has kernel the order~$p$ subgroup $\underline{C_p}$ of constant sequences. 
 
 Define an inverse system $\{((G_n,\sigma),\varphi_{m,n})\}_{n\in\mathbb{N}}$, where  $G_n = C_p^{\mathbb{Z}}$ and $\varphi_{n,n+1} = \varphi$. Put $(G,\tilde\sigma) = {\varprojlim (G_n,\sigma)}$ and let $\varphi_n : G \to C_p^{\mathbb{Z}}$ be the standard projections satisfying $\varphi_n = \varphi_m\circ \varphi_{m,n}$. The pair $(G,\tilde\sigma)$ satisfies the following.
 \begin{enumerate}
 \item If $V$ is an open subgroup of $G$, then there is $n\in\mathbb{N}$ such that $\ker(\varphi_n)\leq V$, whence $\bigcap_{k\in\mathbb{Z}} \tilde\sigma^k(V) \ne \triv$ and $(G,\tilde\sigma)$ does not have finite depth.
 \item Every finite depth quotient of $(G,\tilde\sigma)$ is isomorphic to $(\Cp^{\mathbb{Z}},\sigma)$ and hence has depth~$p$.
 \item If $f\in \bcg{\tilde\sigma}$, then $\varphi_n(f)\in \bcg{\sigma}$ for every $n$, whence $\varphi_n(f)$ has finite support for every $n\in \mathbb{N}$. However, by the definition of $\varphi$ and since $\varphi_n = \phi\circ \varphi_{n+1}$, if $f\ne\ident$ and 
 $$
m = \min(\text{support of }\varphi_{n+1}(f))\text{ and }M = \max(\text{support of }\varphi_{n+1}(f)),
$$ 
then $\min(\text{support of }\varphi_{n}(f)) = m-1$ and $\max(\text{support of }\varphi_{n}(f))=M$, that is, the smallest interval supporting $\varphi_n(f)$ increases as $n$ decreases. That cannot occur if the support of $\varphi_n(f)$ is to be finite and non-empty for every $n$. Hence the support of $\varphi_n(f)$ is empty for every $n$ and $f=\ident$. Therefore $\bcg{\tilde\sigma} = \triv$. 
\end{enumerate} 
\end{example}

Further to Proposition~\ref{prop:con_dense}, the homoclinic subgroup is not in general a complemented subgroup of $(G,\alpha)$, as the next construction shows. 
\begin{example}
\label{ex:does_not_split}
Let $\varphi: C_4^{\mathbb{Z}}\to C_4^{\mathbb{Z}}$ be the homomorphism defined by 
$$
\varphi(x)_n = x_n-x_{n+1}.
$$
Then $\varphi$ commutes with the shift automorphism, $\sigma$. Hence, setting $H_k = C_4^{\mathbb{Z}}$, $\alpha_k = \sigma$  and $\varphi_{k,k+1} = \varphi$ for each $k\in\mathbb{N}$ produces an inverse system $\left\{(H_k, \alpha_k), \varphi_{k,l}\right\}_{k\leq l\in\mathbb{N}}$ whose inverse limit will be denoted
$$
H \dfn
\varprojlim (C_4^{\mathbb{Z}},\sigma) = \left\{ (x_k)\in\left(C_4^{\mathbb{Z}}\right)^{\mathbb{N}} \mid \phi(x_{k+1}) = x_k\right\}.
$$

The subgroup $2C_4$ of $C_4$, which is isomorphic to $C_2$, determines a $\sigma$-stable subgroup $(2C_4)^{\mathbb{Z}}\leq C_4^{\mathbb{Z}}$ that is isomorphic to $C_2^{\mathbb{Z}}$. Since $2C_4^{\mathbb{Z}}$ is also stable under $\varphi$, $H$ contains the $\sigma$-stable subgroup
$$
\varprojlim (2C_4^{\mathbb{Z}},\sigma) = \left\{ (x_k)\in\left(2C_4^{\mathbb{Z}}\right)^{\mathbb{N}} \mid \phi(x_{k+1}) = x_k\right\} \cong \varprojlim (C_2^{\mathbb{Z}},\sigma)
$$
and we also have
$$
\varprojlim (C_4^{\mathbb{Z}},\sigma)/\varprojlim (2C_4^{\mathbb{Z}},\sigma) \cong \varprojlim ((C_4/2C_4)^{\mathbb{Z}},\sigma) 
\cong \varprojlim (C_2^{\mathbb{Z}},\sigma).
$$

Let $\varphi_1: (x_k)\mapsto x_1$ be the projection of $\varprojlim (2C_4^{\mathbb{Z}},\sigma)$ onto its first coordinate. Then the range of $\varphi_1$ is isomorphic to $C_2^{\mathbb{Z}}$ and $\ker \varphi_1$ is a closed, $\sigma$-stable subgroup of~$H$. Put $G = H/\ker\varphi_1$. Then $\varprojlim (2C_4^{\mathbb{Z}},\sigma)/\ker\varphi_1$ is a $\sigma$-stable subgroup of $G$ that is isomorphic to $\varphi_1(2C_4^{\mathbb{Z}})\cong C_2^{\mathbb{Z}}$ and the quotient of $G$ by this subgroup is isomorphic to $\varprojlim (C_2^{\mathbb{Z}},\sigma)$. Therefore 
$$
\overline{\bcg{G}} = \varprojlim (2C_4^{\mathbb{Z}},\sigma)/\ker\varphi_1
$$
and 
$$
G/\overline{\bcg{G}} \cong \varprojlim C_2^{\mathbb{Z}}.
$$
Since both $\overline{\bcg{G}}$ and $C_2^{\mathbb{Z}}$ have exponent~2 and $G$ does not, the sequence
$$
\xymatrix{
\triv  \ar@{->}[r] & \overline{\bcg{G}} \ar@{->}[r] &G \ar@{->}[r] & \varprojlim C_2^{\mathbb{Z}} \ar@{->}[r] &\triv
}
$$
does not split. 
\end{example}

The following corollary to Proposition~\ref{prop:con_dense} is referred to in the next section. 
\begin{corollary}
\label{cor:finte->central}
Let $(G,\alpha)$ be topologically transitive and the subgroup $N\leq G$ be finite, normal and $\alpha$-stable. Then $N$ is contained in the centre of $G$.
\end{corollary}
\begin{proof}
Since $N$ is finite and $\alpha$-stable, for any $x\in N$ there is a positive integer $a$ such that $\alpha^{an}(x) = x$ for all $n\in \mathbb{Z}$. Suppose that $y\in N$ and $h\in \con{\alpha}$. Put $x = hyh^{-1}y^{-1}$ and choose $a$ such that $x = \alpha^{an}(x)$ and $y = \alpha^{an}(y)$ for all $n\in\mathbb{Z}$. Then 
$$
x = \alpha^{an}(x) = \alpha^{an}(h)y\alpha^{an}(h^{-1})y^{-1} \to \ident_G \hbox{ as }n\to\infty.
$$
Hence $y$ is centralized by $\con{\alpha}$ and it follows by Proposition~\ref{prop:con_dense}(\ref{prop:con_densei}) that $y$ belongs to the centre of $G$. 
\end{proof}

\section{A Jordan-H\"older Theorem for Pairs with Finite Depth}
\label{sec:Jordan-Holder}

In this section the algebraic structure of compact pairs $(G,\alpha)$ that are topologically transitive and have finite depth is investigated. The main result is Theorem~\ref{thm:Jordan-Holder}, which establishes that such pairs have a composition series where the factors are isomorphic to shifts, $(F^{\mathbb{Z}},\sigma)$ for some finite simple group $F$, and that the factors are unique up to permutation. This theorem is the direct analogue of the Jordan-H\"older Theorem for finite groups. 

Many of the results of this section have counterparts in the ergodic theory literature. It is shown in \cite[Proposition~2]{Kitchens} that every compact pair $(G,\alpha)$ with finite depth is a \emph{subshift of finite type}, that is, there is a finite group $F$ such that $(G,\alpha)$ is isomorphic to a closed, $\sigma$-stable subgroup of $(F^{\mathbb{Z}},\sigma)$. Indeed, if $V$ is an open, normal subgroup of $G$ with $\bigcap_{n\in\mathbb{Z}} \alpha^n(V) = \triv$, then $F$ may be taken to be $G/V$. Such subshifts are studied in \cite{Kitchens} and a notion of `block size' introduced. The closest counterpart to that notion in the present paper may be seen in Proposition~\ref{prop:prodeq}: the number $k$ such that $\alpha^k(V_+)V_- = G$ is related to the block size. The expansive, or finite depth, pair $(G,\alpha)$ is broken down into factors in \cite[Theorem 1(ii)]{Kitchens} and \cite[Proposition 10.2]{KSchmidt}. These theorems correspond to Proposition~\ref{prop:opennormal} below, which produces an $\alpha$-stable normal series for $(G,\alpha)$ in which each of the factors is a shift. 

\subsection{The depth of $(G,\alpha)$}
\label{sec:more_on_depth}
Analogy with the theory of finite groups motivates the results in this section and the depth of the pair $(G,\alpha)$ corresponds to the order of a group under this analogy. The first step is to verify that the depth of pairs on finite depth, see Definition~\ref{defn:depth}, behaves as expected under quotients. 

\begin{proposition}
\label{prop:findepth}
Let $H$ be a closed, normal, $\alpha$-stable subgroup of $G$. Then $(G,\alpha)$ has finite depth if and only if both $(H,\alpha|_H)$ and $(G/H,\alpha|^{G/H})$ have finite depth.  Moreover, 
\begin{equation}
\label{eq:depth}
\depth(G,\alpha) = \depth(G/H,\alpha|^{G/H})\depth(H,\alpha|_H).
\end{equation} 
\end{proposition}

\begin{proof}
Suppose that  $H$ and $G/H$ have finite depth. Choose open subgroups $W_1\leq G/H$ and $W_2\leq H$ such that 
$$
\bigcap_{k\in{\ZZ}} (\alpha|^{G/H})^k(W_1) =\{ \ident_{G/H}\}\hbox{ and }\bigcap_{k\in{\ZZ}} \alpha^k(W_2) = \{\ident_H\}.
$$ 
Put $V_1 = {\mathfrak q}_H^{-1}(W_1)$, and let $V_2$ be an open subgroup of $G$ such that $V_2\cap H \leq W_2$. Then $V = V_1\cap V_2$ satisfies $\bigcap_{k\in{\ZZ}} \alpha^k(V) = \triv$ and hence $G$ has finite depth. 

For the converse direction, suppose that $G$ has finite depth and let $V$ be an open subgroup such that $\bigcap_{k\in{\ZZ}} \alpha^k(V) = \{\ident_G\}$. Then $H$ has finite depth because 
$$
\bigcap_{k\in{\ZZ}} (\alpha|_H)^k(V\cap H) = \{\ident_H\}.
$$ 
That $G/H$ has finite depth cannot be shown by a similar direct argument, see Remark~\ref{rem:findepth} below. Instead, it will first be shown that $V$ may be assumed to be normal with $V=V_+V_-$ and to satisfy
\begin{equation}
\label{eq:finite_depth2}
\alpha(V_+)\cap V_- = \{\ident_G\} \text{ and } \alpha(V_+)V_-\cap H = \alpha((V\cap H)_+)(V\cap H)_-. 
\end{equation} 

It has already been noted that $V$ may be assumed to be normal. For the other properties note that, since $H$ is $\alpha$-stable, $\alpha(V\cap H) = \alpha(V)\cap H$.
Hence, appealing to Lemma~\ref{lem:recalltidy} and replacing $V$ by $\bigcap_{k=0}^n\alpha^k(V)$ for sufficiently large $n$,  it may be supposed that\footnote{Tidiness above for $V = V_+V_-$ does not imply the same for $V\cap H$, see \cite[Example 6.4]{wi:further}, and Lemma~\ref{lem:recalltidy} must be applied to both $V$ and $V\cap H$ separately.} 
\begin{equation*}
\label{eq:VandHcapV}
V = V_+V_-\text{ and }
V\cap H = (V\cap H)_+(V\cap H)_-.
\end{equation*} 
Then, since $V_+\cap V_- = \{\ident_G\}$, (\ref{eq:finite_depth2}) holds if $V$ is replaced by $\alpha^{-1}(V)\cap V$. All of the subgroups in the above argument are normal because $\alpha$ is an automorphism of $G$.

Assuming that $V$ satisfies (\ref{eq:finite_depth2}), it will be shown that 
\begin{equation}
\label{eq:chosen_subgroup}
\bigcap_{k=m}^n \alpha^k(VH) \leq \left(\bigcap_{k=m}^n \alpha^k(V)\right)H\text{ for all }m,n\in\ZZ,
\end{equation}
whence $\bigcap_{k\in {\mathbb Z}} \alpha^k(VH) = H$ and $(G/H,\alpha|^{G/H})$ has finite depth as claimed.

To establish (\ref{eq:chosen_subgroup}) when $m=0$ and $n=1$, consider $x$ in $VH\cap \alpha(VH)$. Then
\begin{equation}
\label{eq:xincap}
x = v_{1+}v_{1-}y_1 = \alpha(v_{2+})\alpha(v_{2-})y_2, \text{ for some } v_{i\pm} \in V_{\pm} \text{ and }y_i\in H.
\end{equation}
Since $\alpha(V_+)$ and $V_-$ are normal subgroups of $G$, (\ref{eq:finite_depth2})  implies that $\alpha(V_+)V_-$ is a direct product, $\alpha(V_+)\times V_-$. Hence (\ref{eq:xincap}) may be rearranged to yield
$$
 \left(v_{1+}^{-1}\alpha(v_{2+})\right)\left(v_{1-}^{-1}\alpha(v_{2-})\right) = y_1y_2^{-1}, 
$$
where the left side belongs to $\alpha(V_+)V_-$ and the right to $H$. Then (\ref{eq:finite_depth2}) implies that there are $w_{\pm}\in (H\cap V)_{\pm}$ such that 
$$
\left(v_{1+}^{-1}\alpha(v_{2+})\right)\left(v_{1-}^{-1}\alpha(v_{2-})\right) = \alpha(w_+)w_-.
$$
Rearranging this equation yields 
$$
\alpha(w_+)^{-1}\left(v_{1+}^{-1}\alpha(v_{2+})\right) = w_-\left(v_{1-}^{-1}\alpha(v_{2-})\right)^{-1}.
$$
Since the left side belongs to $\alpha(V_+)$ and the right to $V_-$ and $\alpha(V_+)\cap V_- = \{\ident_G\}$, both sides equal the identity and it follows that
$$
\alpha(v_{2+}) = v_{1+} \alpha(w_+),
$$
which belongs to $V_+H$. Substituting into (\ref{eq:xincap}) yields that
$$
x = v_{1+} \alpha(w_+) \alpha(v_{2-})y_2 \in V_+\alpha(V_-)H = \left(V\cap \alpha(V)\right)H
$$ 
and we have shown that $VH\cap \alpha(VH) \leq \left(V\cap \alpha(V)\right)H$. Induction on $n$ and translation by $\alpha^m$ imply that (\ref{eq:chosen_subgroup}) holds for all $m$ and $n$.

Finally, to establish the formula (\ref{eq:depth}) for the depth, choose $V$ satisfying (\ref{eq:finite_depth2}) and recall Definition~\ref{defn:depth} that $\depth(G,\alpha) = [\alpha(V_+) : V_+]$. Then
\begin{equation}
\label{eq:depth1}
\depth(G,\alpha) = [\alpha(V_+) : V_+H\cap \alpha(V_+)][V_+H\cap \alpha(V_+) : V_+]
\end{equation}
because $V_+ \leq V_+H \cap \alpha(V_+)\leq \alpha(V_+)$. Observe first that, since $V_+H$ is a normal subgroup of $G$ and $V_+\leq \alpha(V_+)$, the Second Isomorphism Theorem implies that 
$$
\alpha(V_+) /( V_+H\cap \alpha(V_+)) \cong \alpha(V_+)H/V_+H
$$ 
and consequently that
\begin{equation}
\label{eq:depth2}
[\alpha(V_+) : V_+H\cap \alpha(V_+)] = [\alpha(V_+)H : V_+H] = \depth(G/H,\alpha|^{G/H}),
\end{equation}
where the last equality holds because $(VH/H)_+ = (V_+H)/H$ and, as follows from~(\ref{eq:chosen_subgroup}), $(VH/H)_+\cap (VH/H)_-= \triv$. Observe second that, 
$$
(V_+H\cap \alpha(V_+))/V_+ \cong (H\cap \alpha(V_+))/(H\cap V_+)
$$
and consequently that
\begin{equation}
\label{eq:depth3}
[V_+H\cap \alpha(V_+) : V_+] = [H\cap \alpha(V_+) : H\cap V_+] = \depth(H,\alpha|_H).
\end{equation}
Equations (\ref{eq:depth1}), (\ref{eq:depth2}) and (\ref{eq:depth3}) imply (\ref{eq:depth}).
\end{proof}

\begin{remark}
\label{rem:findepth}
Triviality of $\bigcap_{k\in{\ZZ}} \alpha^k(V)$ does not imply that $\bigcap_{k\in{\ZZ}} (\alpha|^{G/H})^k({\mathfrak q}_H(V))$ is trivial.
Consider $(C_2^{\ZZ},\sigma)$ and the subgroups $V = \left\{f\in C_2^{\ZZ} \mid f(0) = \bar{0} \right\}$ and $H$ the set of constant sequences. Then $\bigcap_{k\in{\ZZ}} \sigma^k(V) = \{\ident_{C_2^{\mathbb{Z}}}\}$ but ${\mathfrak q}_H$ maps $V$ onto $C_2^{\ZZ}/H$. 
\end{remark}

\subsection{An $\alpha$-stable normal series}
\label{sec:alphastableseries}

The first step towards composition series for pairs with finite depth is to show that they have normal series in which the factors are isomorphic to shifts.
\begin{proposition}
\label{prop:opennormal}
Suppose that $(G,\alpha)$ is topologically transitive and has finite depth.  Then $G$ has an $\alpha$-stable series
\begin{equation}
\label{eq:opennormal}
\triv = G_0 \triangleleft G_1 \triangleleft \cdots \triangleleft G_r = G
\end{equation}
of closed, normal subgroups, $G_j$, and, for each $j\in \{0,1,\dots,r-1\}$:
\begin{enumerate}
\item there is a finite group, $F_j$, and a surjective homomorphism 
$$
\varphi_j : F_j^{\ZZ} \to G_{j+1}/G_{j}
$$ 
such that $\ker(\varphi_j) \cap{F_j^{[\ZZ]}} = \triv$; and \label{prop:opennormali}
\item $\varphi_j\circ \sigma = \alpha_j\circ \varphi_j$, where $\alpha_j = \alpha|^{G_{j+1}/G_j}$. 
\end{enumerate}
\end{proposition}
\begin{proof}
Let $V$ be an open, normal subgroup of $G$ such that $\bigcap_{k\in{\ZZ}} \alpha^k(V) = \triv$. Then $V_+$ and $V_-$ are closed, normal subgroups of $G$ and $V_+\cap V_- = \triv$. 

Proposition~\ref{prop:prodeq}(\ref{cor:prodeqii}) implies that $\alpha^{l}(V_+)\cap V_-$ is non-trivial for some $l>0$. Choose $l$ be the smallest such integer and put $F_0 := \alpha^{l}(V_+)\cap V_-$. Then $F_0$ is a finite normal subgroup, by Lemma~\ref{lem:capfinite}, and $\alpha^{l-1}(V_+) \times F_0 \times \alpha(V_-)$ is a direct product of closed, normal subgroups of $G$. Since $\alpha^n(F_0)$ is contained in $\alpha(V_-)$ when $n>0$ and in $\alpha^{l-1}(V_+)$ when $n<0$, it follows that the map $\varphi : F_0^{[\ZZ]}\to G$ defined by
$$
\varphi(f) = \prod_{n\in\ZZ} \alpha^n({f(n)}),\qquad (f\in F_0^{[\ZZ]}), 
$$  
is injective. Since  $\left\{\alpha^{l-n}(V_+)\alpha^n(V_-)\right\}_{n\in\NN}$ is a base of neighbourhoods of $\ident_G$, and since 
$$
\varphi\left( \left\{ f\in F_0^{[\ZZ]} \mid f(m) = \ident \text{ if } -n\leq m\leq n\right\}\right) \leq \alpha^{l-n}(V_+)\alpha^n(V_-)
$$ 
for each $n$, this map extends uniquely to a continuous homomorphism ${\varphi}_0 : F_0^{\ZZ}\to G$ such that $\ker(\varphi_0) \cap F_0^{[\ZZ]} = \triv$.  By construction, $\alpha\circ{\varphi}_0 ={\varphi}_0\circ \sigma$ and  $G_1 \dfn {\varphi}_0(F_0^{{\mathbb Z}})$ is a closed, normal $\alpha$-stable subgroup of $G$. 

By Proposition~\ref{prop:findepth}, $(G/G_1,\alpha|^{G/G_1})$ has finite depth. Hence the argument may be repeated for $(G/G_1,\alpha|^{G/G_1})$ (provided that it is non-trivial) to produce a closed, normal, $\alpha$-stable subgroup of $G/G_1$ that is isomorphic to a quotient of~$F_1^{\mathbb{Z}}$ for some finite group $F_1$. Pulling back to $G$ yields the subgroup, $G_2$. Iterating produces subgroups $G_j$, $j=1, 2, \dots$ as required. This iteration terminates at some $r$ with $G_r=G$ because Proposition~\ref{prop:findepth} and the fact that $\depth(G_{j+1}/G_j,\alpha_j) = |F_j|$ imply that the sequence of integers $\{\depth(G/G_j,\alpha|^{G/G_j}\}$ strictly decreases so long as $G_j\ne G$. 
\end{proof}

\begin{remark}
The factor in Proposition~\ref{prop:opennormal}(\ref{prop:opennormali}) is a quotient of a shift. It will be seen shortly however that, although the homomorphism $\varphi_j$ need not be an isomorphism, the factor is in fact isomorphic to a shift.  
\end{remark}

The calculation in the previous paragraph yields a formula for the depth of $G$. 
\begin{corollary}
\label{cor:depth}
The depth of the topologically transitive pair $(G,\alpha)$ in Proposition~\ref{prop:opennormal} is equal to $\prod_{j=0}^{r-1} |F_j|$.
\end{corollary}

%% Change k to q for consistency with application in Section 6. 
\subsection{Normal subgroups of $(F^{\mathbb{Z}},\sigma)$}
\label{sec:technical_lemmas}

The factors in the $\alpha$-stable normal series (\ref{eq:opennormal}) are shifts $(F^{\mathbb{Z}},\sigma)$, and the next aim is to analyze the structure of these shifts, particularly in the case when $F$ is simple. 
\begin{lemma}
\label{lem:Simple_normal}
Suppose that $H$ is a closed, normal, $\sigma$-invariant subgroup of $F^{\ZZ}$, where $F$ is a finite simple group.  Then exactly one of the following holds:
\begin{enumerate}
\item \label{lem:Simple_normali}
$H = F^{\ZZ}$;
\item  \label{lem:Simple_normalii}
$H = \triv$; 
\item \label{lem:Simple_normaliii}
there is $n\in {\mathbb Z}$ such that $H = \left\{ h\in F^{\ZZ} \mid h(j) = \ident \text{ for }j\geq n\right\}$; 
\end{enumerate}
if \emph{\bf $F$ is abelian}, $F = \Cp$ is the cyclic group of order $p$ for some prime $p$, \\ there are $d\in{\mathbb N}$ and  $a_0,\ \dots,\ a_{d-1} \in \Cp$ such that 
\begin{enumerate}
\setcounter{enumi}{3}
\item \label{lem:Simple_normaliv}
$H$ is the subgroup of functions satisfying the order~$d$ difference equation
$$
H = \Bigl\{ h\in F^{\ZZ} \mid h(r+d) = \sum_{j=0}^{d-1} a_jh(r+j)\text{ for all }r\in{\mathbb Z} \Bigr\},
$$ 
in which case $H$ is a finite group with order $p^d$; or 
\item \label{lem:Simple_normalv}
there is $n\in {\mathbb Z}$ such that 
$$
H = \Bigl\{ h\in F^{\ZZ} \mid h(r+d) = \sum_{j=0}^{d-1} a_jh(r+j)\text{ for all }r\geq  n\Bigr\}.
$$
\end{enumerate}
\end{lemma}
\begin{proof}
Suppose that, for each $n\in {\mathbb Z}$ and $d\geq0$, there is $h\in H$ with
\begin{equation*}
\label{eq:independent}
h(j) = \ident_F\text{ for every }j\in\{n,  n+1, \dots, n+d-1\} \text{ and } h(n+d) \ne \ident_F.
\end{equation*}
Then, since $H$ is normal in $F^{\mathbb Z}$ and $F$ is simple, the homomorphism
$$
\pi_{n,d} : H \to F^{d+1}, \quad h \mapsto \left(h(n),h(n+1), \dots, h(n+d) \right)
$$ is surjective for each $n\in {\mathbb Z}$ and $d\geq0$. Since $H$ is closed, this implies that (\ref{lem:Simple_normali}) holds. 
Therefore, if (\ref{lem:Simple_normali}) fails, there are $n^*\in {\mathbb Z}$ and $d\geq0$ such that for every $h\in H$
\begin{equation*}
h(j) = \ident_F\text{ for }j\in\{n^*, n^*+1, \dots, n^*+d-1\} \implies h(n^*+d) = \ident_F. 
\end{equation*}
Then $\alpha$-invariance of the subgroup $H$ implies that, for every $h\in H$ and $r\geq n^*$,
\begin{equation}
\label{eq:statement}
\hbox{ the values $h(j)$ for $j\in\{r,  r+1, \dots, r+d-1\}$ determine $h(r+d)$.}
\end{equation}

Assume that (\ref{lem:Simple_normali}) fails and consider the case when $F$ is non-abelian. Then, since $H\triangleleft F^{\mathbb Z}$, conjugation of $H$ by elements $[f]_{r+d} \in F^{\mathbb Z}$ defined by  
$$
[f]_{r+d}(j) = \begin{cases} f & \text{ if } j = r+d\\
\ident_F & \text{ if } j\ne r+d
\end{cases},\quad (f\in F),
$$ 
leads to a contradiction to (\ref{eq:statement}) unless $h(r+d) = \ident_F$ for every $r\geq n^*$. Therefore~$H$ satisfies either (\ref{lem:Simple_normalii}) or (\ref{lem:Simple_normaliii}) when $F$ is not abelian. 

Assume that (\ref{lem:Simple_normali}), (\ref{lem:Simple_normalii}) and (\ref{lem:Simple_normaliii}) fail and let $F = \Cp$ be abelian. Then (\ref{eq:statement}) and the $\sigma$-invariance of $H$ imply that there are $a_0,\ \dots, \, a_{d-1}$ in $\Cp$ such that every $h\in H$ satisfies
\begin{equation}
\label{eq:diff1}
h(r+d) = \sum_{j=0}^{d-1} a_jh(r+j)\text{ for all }r \geq  n^*.
\end{equation}
($\ast$) \emph{Choose $d$ to be the smallest value for which {\rm (\ref{eq:diff1})} holds.}\\
Should it happen that (\ref{eq:diff1}) holds for every $n^*$, then either case (\ref{lem:Simple_normalii}) or (\ref{lem:Simple_normaliv}) obtains. Otherwise,\\ 
($\ast\ast$) \emph{choose the smallest value of  $n^*$ for which {\rm(\ref{eq:diff1})} holds}.\\ Then (\ref{lem:Simple_normalv}) obtains with $n=n^*$: it may be seen immediately that $H$ is contained in the set on the right of the equation and it may be shown that ($\ast$) and ($\ast\ast$) imply that $H$ is no smaller than this set. 
\end{proof}

%\begin{remark}
%\label{rem:Simple_normal}
%Additional cases may occur in Lemma~\ref{lem:Simple_normal} if $H$ is not required to be normal. Case~(\ref{lem:Simple_normaliv}), where $H$ consists of periodic sequences, also occurs for non-abelian groups. Furthermore, $\alpha$-invariant subgroups of $F^{\mathbb Z}$ also arise from each increasing sequence $\triv < F_1 < \dots < F_s = F$ of subgroups of $F$.  These are more complicated to describe but, fortunately, their description is not required here.
%\end{remark}

\begin{lemma}
\label{lem:kernel}
Let $F$ be a finite group, and suppose that $N$ is a closed, normal $\sigma$-stable subgroup of $F^{\mathbb Z}$ such that $N\cap F^{[{\mathbb Z}]} = \triv$. Then $N$ is finite and contained in the centre of $F^{\mathbb Z}$.
\end{lemma}
\begin{proof}
The finite group $F$ has, by the Jordan-H\"older Theorem, a composition series
\begin{equation}
\label{eq:finitecomp}
\triv = F_0 \triangleleft F_1 \triangleleft \cdots \triangleleft F_{r-1} \triangleleft F_r = F,
\end{equation}
and all composition series for $F$ have the same length, $r$. The proof that $N$ is finite will be by induction on~$r$. For the base case, note that $N$ is trivial when $r=0$.

Let $F$ have a composition series of length $r>0$ and assume that the claim has been established for groups with shorter composition series. Let $N$ be a closed, normal, $\sigma$-stable subgroup of $F^{\mathbb Z}$ such that $N\cap F^{[{\mathbb Z}]} = \triv$. Then, considering $F_{r-1}^{\mathbb Z}$ to be a subgroup of $F^{\mathbb Z}$: $N\cap F_{r-1}^{\mathbb Z}$ is a closed, normal, $\sigma$-stable of $F_{r-1}^{\mathbb Z}$ and $N\cap F_{r-1}^{[{\mathbb Z}]}$ is trivial. Hence, by the induction hypothesis, 
\begin{equation}
\label{eq:InductionHyp}
N\cap F_{r-1}^{{\mathbb Z}}\text{ is finite}.
\end{equation} 

The Second Isomorphism Theorem implies that 
\begin{equation}
\label{eq:2ndIso}
N/\left(N\cap F_{r-1}^{{\mathbb Z}}\right)\cong NF_{r-1}^{\mathbb Z}/F_{r-1}^{\mathbb Z},
\end{equation}
which is a closed, normal, $\sigma|^{F^{\mathbb Z}/F_{r-1}^{\mathbb Z}}$-stable subgroup of $F^{\mathbb Z}/F_{r-1}^{\mathbb Z}$. The map 
$$
\varphi : F^{\mathbb Z}/F_{r-1}^{\mathbb Z}\to \left(F/F_{r-1}\right)^{\mathbb Z},\ \ \varphi(fF_{r-1}^{\mathbb Z})(j) = f(j)F_{r-1}
$$ 
is an isomorphism that  intertwines $\sigma$. Hence $\varphi(NF_{r-1}^{\mathbb Z}/F_{r-1}^{\mathbb Z})$ is a closed, normal, $\sigma$-stable subgroup of $(F/F_{r-1})^{\mathbb Z}$, where $F/F_{r-1}$ is a finite simple group because (\ref{eq:finitecomp}) is a composition series. Hence, according to Lemma~\ref{lem:Simple_normal},   $\varphi(NF_{r-1}^{\mathbb Z}/F_{r-1}^{\mathbb Z})$ falls into one of five cases. Since cases (\ref{lem:Simple_normaliii}) and (\ref{lem:Simple_normalv}) are $\sigma$-invariant but not $\sigma$-stable, $\varphi(NF_{r-1}^{\mathbb Z}/F_{r-1}^{\mathbb Z})$ must in fact be either the whole group (case~(\ref{lem:Simple_normali})) or finite (cases~(\ref{lem:Simple_normalii}) and~(\ref{lem:Simple_normaliv})). 

It will be shown that 
$$
\varphi(NF_{r-1}^{\mathbb Z}/F_{r-1}^{\mathbb Z})\cap (F/F_{r-1})^{[{\mathbb Z}]} = \triv,
$$ 
which excludes the case when $\varphi(NF_{r-1}^{\mathbb Z}/F_{r-1}^{\mathbb Z})$ is the whole group. To this end, let $f\in N$ be such that $\varphi(fF_{r-1}^{\mathbb Z})$ belongs to  $(F/F_{r-1})^{[{\mathbb Z}]}$. Since $N$ is compact,  $\{\sigma^k(f)\}_{k\geq0}$ has an accumulation point, $g$ say, which belongs to $N\cap F_{r-1}^{\mathbb Z}$ because
$$
\varphi(\sigma^k(f)F_{r-1}^{\mathbb Z}) =  \sigma^k(\varphi(fF_{r-1}^{\mathbb Z})) \to \ident\text{ as }k\to \infty.
$$ 
 Since  $N\cap F_{r-1}^{\mathbb Z}$ is a finite $\sigma$-stable subgroup by (\ref{eq:InductionHyp}), $g$ is periodic. Hence, if the  period of $g$ is $p_1$, there is an integer $n_1$ such that $f|_{(-\infty, n_1]}$ is $p_1$-periodic. Similarly, there are $p_2\geq1$ and $n_2\in {\mathbb Z}$ such that $f|_{[n_2,\infty)}$ is $p_2$-periodic, and it follows that $f^{-1}\sigma^{p_1p_2}(f)$ belongs to $N$ and has finite support. Since $N\cap F^{[{\mathbb Z}]} = \triv$, we conclude that $f = \beta^{p_1p_2}(f)$ and hence that $f=g$ and thus belongs to $F_{r-1}^{\mathbb Z}$. Therefore, as claimed, $\varphi$ does not map $NF_{r-1}^{\mathbb Z}/F_{r-1}^{\mathbb Z}$ onto $(F/F_{r-1})^{{\mathbb Z}}$  and $NF_{r-1}^{\mathbb Z}/F_{r-1}^{\mathbb Z}$ is finite. The isomorphism in (\ref{eq:2ndIso}) then implies that $N/\left(N\cap F_{r-1}^{{\mathbb Z}}\right)$ is finite and, combining with (\ref{eq:InductionHyp}), it follows that $N$ is finite. 
 
Since $N$ is finite and $\sigma$-stable and $F^{\mathbb{Z}}$ is topologically transitive,  Corollary~\ref{cor:finte->central} shows that it is contained in the centre of $F^{\mathbb{Z}}$.
\end{proof}

The factors in the composition series for $(G,\alpha)$ given in Theorem~\ref{thm:Jordan-Holder} below cannot be decomposed any further in the following sense defined in ergodic theory, see~\cite{KitchSchmidt2} for example.

\begin{definition}
\label{defn:simplecomp}
The pair $(G,\alpha)$, with $G$ compact and infinite, is \emph{irreducible} if every proper, closed, normal, $\alpha$-stable subgroup of $G$ is finite. 
\end{definition}

As might be expected, irreducibility of $(F,\sigma)$ corresponds to simplicity of $F$.
%Delete proof because in ergodic theory literature? 
\begin{proposition}
\label{prop:simplecomp}
The pair $(G,\alpha)$ is irreducible if and only if it is isomorphic to $(F^{\ZZ},\sigma)$ for some finite simple group  $F$.
\end{proposition}
\begin{proof}
Suppose that $G = F^{\mathbb{Z}}$, where $F$ is a finite simple group and let $H$ be a closed, normal $\alpha$-stable subgroup. Consider first the case when $F$ is not abelian. If $H$ is not trivial, then there is $n\in\mathbb{Z}$ and $h\in H$ such that $h(n)\ne \ident_F$. Then, since $F$ is not abelian, $\{ [h, f^{[n]}] \mid f\in F\}$ contains a non-identity element and it follows, since $F$ is simple, that $H$ contains $F^{[n]} := \{ f^{[n]} \mid f\in F\}$. That $H=G$ follows because $H$ is shift-invariant and closed. Consider next the case when $F$ is abelian, that is, $F$ is a cyclic group of order $p$ for some prime $p$. Suppose that $H$ contains a non-identity element, $h$, with finite support. Let $n$ be the minimum of the support of $h$. Then, by adding suitable shifts of powers of $h$, it may be seen that for each $N>0$ there is an element $h'\in H$ whose support in the interval $[n-N,n+N]$ is $\{n\}$. Since $H$ is closed and a subgroup, it follows that the elements $\bar{i}^{[n]}$ ($\bar{i}\in C_p$) belong to $H$. Shift-invariance and closedness of $H$ then imply that $H$ is equal to~$G$. Hence, if $H$ is a proper subgroup of $G$, it does not intersect $F^{[\mathbb{Z}]}$ non-trivially and Lemma~\ref{lem:kernel} implies that $H$ is finite.

For the converse, assume that $(G,\alpha)$ is irreducible. Then $(G,\alpha)$ has finite depth, by Proposition~\ref{prop:Proj_Limit}. Proposition~\ref{prop:opennormal} implies that $G$ has an $\alpha$-stable series of the form (\ref{eq:opennormal}) which, since $(G,\alpha)$ is irreducible, must have $G=G_1$. Hence $G = \varphi(F^{\ZZ})$ for some finite group $F$ and  $\ker(\varphi)\cap F^{[\mathbb{Z}]} = \triv$. By Lemma~\ref{lem:kernel}, $\ker(\varphi)$ is finite and contained in the centre of $F^{\ZZ}$. 
Irreducibility of $(G,\alpha)$ then further implies that $(F^{\mathbb{Z}},\sigma)$ is irreducible, whence $F$ is simple. Lemma~\ref{lem:Simple_normal} then applies to show that either $\ker(\varphi)$ is trivial or $F = \Cp$ is abelian and
$$
\ker(\varphi) = \Bigl\{ h\in F^{\ZZ} \mid h(r-k) = \sum_{j=0}^{k-1} a_jh(r-j)\text{ for all }r\in{\mathbb Z} \Bigr\},
$$ 
for some $a_0,\ \dots,\ a_{k-1} \in \Cp$. The proof is complete in the first case. In the second case, the surjective map $\psi : F^{\ZZ}\to F^{\ZZ}$ defined by
$$
\psi(h)(r) = \sum_{j=0}^{k-1} a_jh(r-j)
$$
commutes with $\beta$ and $\ker(\psi) = \ker(\varphi)$. Thus $G\cong F^{\ZZ}/\ker(\varphi) \cong F^{\ZZ}$. 
\end{proof}

The proposition implies the assertion in the next corollary that $\varphi$ does not have a right inverse because the range of any such inverse would a proper, infinite subgroup of $C_2^{\mathbb{Z}}$ that is stable under the shift. All other assertions are obvious. 
\begin{corollary}
\label{cor:not_complemented}
Let $\varphi : C_2^{\mathbb{Z}}\to C_2^{\mathbb{Z}}$ be the homomorphism $\varphi(f)_n = f_n-f_{n+1}$. Then $\varphi$ commutes with the shift automorphism and $C:= \ker\varphi$ is the order~2 subgroup of constant functions in $C_2^{\mathbb{Z}}$. The sequence 
$$
\xymatrix{
\triv  \ar@{->}[r] & C \ar@{->}[r] &C_2^{\mathbb{Z}}  \ar@{->}^{\varphi}[r] & C_2^{\mathbb{Z}} \ar@{->}[r] &\triv
}
$$ 
is exact and does not split. 
\endproof
\end{corollary}

\begin{remark}
\label{rem:simple}
An alternative notion of irreducibility of a topologically transitive pair $(G,\alpha)$ might be that every proper subgroup of $G$ has trivial nub. This weaker definition implies that $(G,\alpha)\cong (F^{\ZZ},\beta)$ in the non-abelian case, but in the abelian case encompasses groups such as those described in Example~\ref{ex:con_2-ways_not_dense}. 
%%(2) If $(G,\alpha)$ is simple and $G$ commutative, then $G$ does have finite $\alpha$-stable subgroups. It follows from  Corollary~\ref{cor:simplecomp} that, if $N$ is such a finite subgroup, then $(G/N,\alpha|^{G/N})$ is isomorphic to $(G,\alpha)$. 
\end{remark}

\subsection{The Jordan-H\"older and Schreier Refinement Theorems}
Transferring the Schreier Refinement and Jordan-H\"older Theorems from finite groups to pairs $(G,\alpha)$ with finite depth uses a version of the Zassenhaus Lemma, see \cite[Lemma~3.3]{Lang:Algebra}. The statement of this lemma does not apply without change however because the intersection of topologically transitive $\alpha$-stable subgroups of $G$ need not be topologically transitive, see Example~\ref{ex:intersection_not_ergodic}.  The next couple of definitions are required for the modified statement of the lemma. 

\begin{definition}
\label{defn:meet_minstable}
Let $(G, \alpha)$ be a compact group with automorphism $\alpha$. The \emph{nub intersection} of closed, $\alpha$-stable subgroups  $H_1$ and $H_2$ of $G$ is defined to be
$$
H_1\wedge H_2 = \nub{\alpha|_{H_1\cap H_2}}.
$$
\end{definition}

\begin{definition}
\label{def:isoms}
The pairs $(G_1,\alpha_1)$ and $(G_2,\alpha_2)$ are \emph{co-commensurable} if there is a pair $(G,\alpha)$ and surjective homomorphisms $\varphi_i : G\to G_i$ such that $\varphi_i\circ \alpha_i = \alpha\circ \varphi_i$ with $\ker\varphi_i$ a finite subgroup of $G$, for $i=1,2$. Denote this as $(G_1,\alpha_1) \isoms (G_2,\alpha_2)$. 
\end{definition}

\begin{remark}
\label{rem:qiso-simple}
It follows from Proposition~\ref{prop:simplecomp} that two irreducible pairs, $(G_1,\alpha_1)$ and $(G_2,\alpha_2)$, that are co-commensurable are in fact isomorphic. 
\end{remark}

\begin{proposition}
\label{prop:quasi-iso}
\begin{enumerate}
\item Co-commensurability is an equivalence relation.
\label{prop:quasi-isoi}
\item 
\label{prop:quasi-isoii}
The pairs $(H_1,\alpha_1)$ and $(H_2,\alpha_2)$ are co-commensurable if and only if there is a pair $(L,\gamma)$ and surjective homomorphisms $\psi_i : H_i\to L$ with $\gamma\circ \psi_i = \psi_i\circ \alpha_i$ for $i=1,2$, and with $\ker(\psi_i)$ finite.
\end{enumerate}
\end{proposition}
\begin{proof}
(\ref{prop:quasi-isoi}) Reflexivity and symmetry follow immediately from the definition. Suppose that $(G_1,\alpha_1)$, $(G_2,\alpha_2)$ and $(G_3,\alpha_3)$ are  pairs with 
$$
(G_1,\alpha_1)\isoms (G_2,\alpha_2)\text{ and }(G_2,\alpha_1)\isoms (G_3,\alpha_3).
$$ 
Transitivity will follow by showing that there is a group $G_{13}$ and surjective homomorphisms $\pi_{12} : G_{13} \to G_{12}$ and $\pi_{23} : G_{13}\to G_{23}$ having finite kernels that complete the diagram:
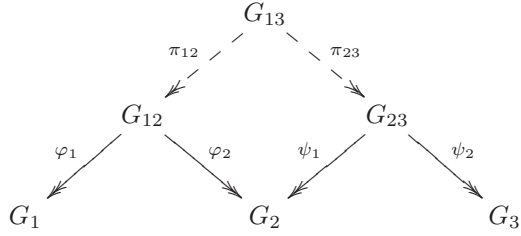
\begin{equation*}
\xymatrix{
&&G_{13}  \ar@{-->>}[dl]_{\pi_{12}} \ar@{-->>}[dr]^{\pi_{23}} && \\
&G_{12} \ar@{>>}[dl]_{\varphi_1} \ar@{>>}[dr]^{\varphi_2}&&G_{23} \ar@{>>}[dl]_{\psi_1} \ar@{>>}[dr]^{\psi_2}&\\
G_1&&G_2&&G_3\\
}
\end{equation*}
The existence of this group and homomorphisms is proved by applying part~(\ref{prop:quasi-isoii}) with  $H_1 = G_{12}$, $H_2 = G_{23}$ and $L = G_2$. 

(\ref{prop:quasi-isoii}) For the `if' direction, let $(H_i, \alpha_i)$ and $(L,\gamma)$ be pairs such that there are surjective homomorphisms $\psi_i : H_i\to L$, $i=1,2$, that intertwine $\alpha_i$ and $\gamma$. Put
$$
H := \left\{(x_{1},x_{2})\in H_{1}\times H_{2} \mid \psi_1(x_{1}) = \psi_2(x_{2})\right\},\ \alpha = (\alpha_1\times\alpha_2)|_H,
$$
and let $\varphi_{i} : H \to H_{i}$ be the coordinate projections. Then: $\varphi_{1}$ and $\varphi_{2}$ are surjective because $\psi_1$ and $\psi_2$ are; have $\ker\varphi_1\cong\ker\psi_2$ and $\ker\varphi_2\cong\ker\psi_1$; and complete the diagram on the left of Figure~\ref{co-commensurate}. Hence $(H_1,\alpha_1) \isoms (H_2,\alpha_2)$.

\begin{figure}[htbp]
\begin{center}
\begin{equation*}
\xymatrix{
&H  \ar@{-->>}[dr]^{\varphi_2}&&&&H  \ar@{>>}[dr]^{\varphi_2} & \\
H_1 \ar@{<<--}[ur]^{\varphi_1} \ar@{>>}[dr]_{\psi_1} && H_2  \ar@{>>}[dl]^{\psi_2} && H_1 \ar@{<<-}[ur]^{\varphi_1} \ar@{-->>}[dr]_{\psi_1} && H_2  \ar@{-->>}[dl]^{\psi_2}\\ 
& L &&&& L & \\
}
\end{equation*}\caption{An equivalent formulation of co-commensurability}
\label{co-commensurate}
\end{center}
\end{figure}
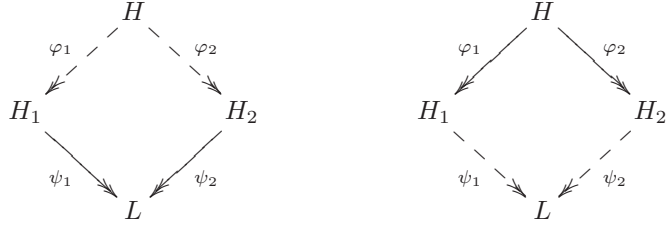

For the `only if' direction, suppose that $(H_1,\alpha_1) \isoms (H_2,\alpha_2)$ and let $(H,\alpha)$ and $\varphi_i : H\to H_i$ are a  group and homomorphisms giving effect to their co-commensurability. Put  
$$
L := H/(\ker(\varphi_1)\ker(\varphi_2)),\quad \gamma = \alpha|^{L},
$$ 
and define $\varphi_i := \rho_i\circ \eta_i$, where $\eta_i : H_i \to H/\ker(\varphi_i)$ is given by the First Isomorphism Theorem, and $\rho_i :  H/\ker(\varphi_i)\to H/(\ker(\varphi_1)\ker(\varphi_2))$ is the obvious map. Then the pair $(L,\gamma)$ and homomorphisms $\psi_i : H_i \to L$ complete the diagram on the right of Figure~\ref{co-commensurate}.
\end{proof}

One further fact, which is easily verified, is required for the proof of the Zassenhaus Lemma.
\begin{lemma}
\label{lem:prodminstable}
Let $G$ be a compact group and $\alpha\in \Aut(G)$. Suppose that $L$ and $H$ are closed, $\alpha$-stable, topologically transitive subgroups of $G$, and that $L$ is normalized by $H$. Then $LH$ is a closed,  $\alpha$-stable, topologically transitive subgroup of $G$.
\endproof
\end{lemma}

The Zassenhaus Lemma may now be adapted from \cite[Lemma~3.3]{Lang:Algebra}.
\begin{lemma}
\label{lem:Zassenhaus}
Let $(G,\alpha)$ be a topologically transitive pair with finite depth. Suppose, for $i=1,2$, that $H_i$ is a closed, $\alpha$-stable subgroup of $G$ and that $L_i$ is a  closed, normal, $\alpha$-stable, topologically transitive subgroup of $H_i$. Then 
\begin{eqnarray*}
L_1(H_1\wedge L_2) \text{ is normal in  } L_1(H_1\wedge H_2),\\
(L_1\wedge H_2)L_2 \text{ is normal in  }(H_1\wedge H_2)L_2,
\end{eqnarray*}
and all groups are $\alpha$-stable and topologically transitive. Denote the automorphisms induced on the respective factor groups by $\alpha_1$ and $\alpha_2$. Then:
\begin{enumerate}
\item the factor pairs are co-commensurable, that is,
\begin{equation*}
\label{lem:Zassenhaus1}
({L_1(H_1\wedge H_2)}/{L_1(H_1\wedge L_2)},\alpha_1) \isoms ({(H_1\wedge H_2)L_2}/{(L_1\wedge H_2)L_2},\alpha_2); and
\end{equation*}
\item   are topologically transitive. \label{lem:Zassenhaus2}
\end{enumerate}
\end{lemma}
\begin{proof}
Lemma~\ref{lem:prodminstable} implies that all groups are $\alpha$-stable and topologically transitive. The group $L_1(H_1\wedge L_2)$ is the largest $\alpha$-stable, topologically transitive subgroup of $L_1(H_1\cap L_2)$. Since, by \cite[Lemma~3.3]{Lang:Algebra}, the latter group is normalized by $L_1(H_1\wedge H_2)$, which is also $\alpha$-stable, it follows that $L_1(H_1\wedge L_2)$ is normalized by $L_1(H_1\wedge H_2)$ as claimed. That $(L_1\wedge H_2)L_2$ is normal in  $(H_1\wedge H_2)L_2$ may be seen similarly. 
Proposition~\ref{prop:ErgQuot} implies that $L_1(H_1\wedge H_2)/L_1(H_1\wedge L_2)$ and $(H_1\wedge H_2)L_2/(L_1\wedge H_2)L_2$ are topologically transitive as well, so that (\ref{lem:Zassenhaus2}) holds.

It remains to prove (\ref{lem:Zassenhaus1}), and this will be deduced from the isomorphism
\begin{equation}
\label{eq:Zassenhaus}
\textstyle{\frac{L_1(H_1\cap H_2)}{L_1(H_1\cap L_2)}} \cong \textstyle{\frac{(H_1\cap H_2)L_2}{(L_1\cap H_2)L_2}},
\end{equation}
which is shown in \cite[Lemma~3.3]{Lang:Algebra}.

Since each open $\alpha$-stable subgroup of $L_1(H_1\cap H_2)/L_1(H_1\cap L_2)$ pulls back to an open $\alpha$-stable subgroup of $L_1(H_1\cap H_2)$, it follows that
$$
\nub{\textstyle{\frac{L_1(H_1\cap H_2)}{L_1(H_1\cap L_2)} }}  = \textstyle{ \frac{(L_1(H_1\wedge H_2))(L_1(H_1\cap L_2))}{L_1(H_1\cap L_2)}},
$$
and hence, by the Second Isomorphism Theorem, that
$$
\nub{\textstyle{\frac{L_1(H_1\cap H_2)}{L_1(H_1\cap L_2)} }}  \cong  \textstyle{\frac{L_1(H_1\wedge H_2)}{L_1(H_1\wedge H_2)\cap L_1(H_1\cap L_2)}}.
$$
Applying the same argument to the right hand side of (\ref{eq:Zassenhaus}) yields that 
$$
\nub{\textstyle{\frac{(H_1\cap H_2)L_2}{L_1\cap H_2)L_2} }}  \cong  \textstyle{\frac{(H_1\wedge H_2)L_2}{(H_1\wedge H_2)L_2\cap (L_1\cap H_2)L_2}},
$$
whence, from  (\ref{eq:Zassenhaus}),
$$
\textstyle{\frac{L_1(H_1\wedge H_2)}{L_1(H_1\wedge H_2)\cap L_1(H_1\cap L_2)} }\cong \textstyle{\frac{(H_1\wedge H_2)L_2}{(H_1\wedge H_2)L_2\cap (L_1\cap H_2)L_2}}.
$$
Since $L_1(H_1\wedge L_2)$ has finite index in $L_1(H_1\cap L_2)$ and is contained in $L_1(H_1\wedge H_2)$, the left side of this equation is a quotient of $L_1(H_1\wedge H_2)/L_1(H_1\wedge L_2)$ by a finite, normal, $\alpha$-stable subgroup.  Similarly, the right side is a quotient of $(H_1\wedge H_2)L_2/(L_1\wedge H_2)L_2$ by a finite  subgroup. Therefore, by Proposition~\ref{prop:quasi-iso}, 
\begin{equation*}
\textstyle{\frac{L_1(H_1\wedge H_2)}{L_1(H_1\wedge L_2)}} \isoms \textstyle{\frac{(H_1\wedge H_2)L_2}{(L_1\wedge H_2)L_2}}.
\end{equation*}
\end{proof}

Versions of the Schreier Refinement Theorem, \cite[Theorem~3.4]{Lang:Algebra}, and the Jordan-H\"older Theorem, \cite[Theorem~3.5]{Lang:Algebra}, for topologically transitive pairs with finite depth follow from Lemma~\ref{lem:Zassenhaus} exactly as they do for finite groups from the Zassenhaus Lemma. Some terminology is needed in order to state these theorems.

\begin{definition}
\label{defn:subnormal_minstable}
Let $(G,\alpha)$ be topologically transitive. 
\begin{enumerate}
\item \label{defn:subnormal_minstablei}
An increasing sequence,
\begin{equation*}
\label{eq:subnormal_minstable}
\triv = G_0 \leq G_1 \leq \cdots \leq G_{r-1} \leq G_r = G,
\end{equation*}
 of $\alpha$-stable, topologically transitive subgroups of $G$ is a  \emph{subnormal series for $(G,\alpha)$} if $G_{j-1}\triangleleft G_{j}$ and $\alpha|_{G_j}^{G_j/G_{j-1}}$ is topologically transitive on ${G_j}^{G_j/G_{j-1}}$ for each $j\in\{1,2,\dots,r\}$. The \emph{$j$th subquotient} of the series is the pair $(G_j/G_{j-1},\alpha_j)$, where $\alpha_j = \alpha|_{G_j}^{G_j/G_{j-1}}$. 
 \item Two subnormal series 
\begin{align*}
\triv = G_0 \triangleleft G_1 &\triangleleft \cdots \triangleleft G_{r-1} \triangleleft G_r = G,\\
\triv = H_0 \triangleleft H_1 &\triangleleft \cdots \triangleleft H_{s-1} \triangleleft H_s = G,
\end{align*}
of $\alpha$-stable, topologically transitive subgroups of $G$ are \emph{equivalent} if $r=s$ and there is a permutation, $\pi$, of $\{1,2,\dots,r\}$ such that 
$$
(G_{j}/G_{j-1},\alpha_j) \isoms (H_{\pi(j)}/H_{\pi(j)-1},\alpha_{\pi(j)})\text{ for each }j.
$$ 
\item A \emph{refinement} of a subnormal series is subnormal series in which the same subgroups and possibly some others appear.
\item A \emph{composition series} for $(G,\alpha)$ is a subnormal series, as in (\ref{defn:subnormal_minstablei}), for which each subquotient $(G_j/G_{j-1},\alpha_j)$ is an irreducible pair. 
\end{enumerate}
\end{definition}

\begin{theorem}
\label{thm:Schreier}
Let $(G,\alpha)$ be a topologically transitive pair with finite depth. Then any two subnormal series of $\alpha$-stable, topologically transitive subgroups of $G$ have equivalent refinements. \endproof
\end{theorem}

\begin{theorem}
\label{thm:Jordan-Holder}
Let $(G,\alpha)$ be a topologically transitive pair with finite depth. Then~$G$ has a composition series. Any two composition series for $(G,\alpha)$ are equivalent.
\endproof
\end{theorem}

\begin{remark}
\label{rem:Jordan-Holder}
{\it (a)\/} Remark~\ref{rem:qiso-simple} implies that the sub-quotients appearing in any two composition series for $(G,\alpha)$ are in fact, up to a permutation, isomorphic. \label{rem:Jordan-Holderi}\\
{\it (b)\/} Suppose that $(G_1,\alpha_1)$ and $(G_2,\alpha_2)$ are co-commensurable and have finite depth, and let $(G,\alpha)$ and the  homomorphisms $\varphi_i : G\to G_i$ be as in Definition~\ref{def:isoms}. Let 
\begin{equation*}
\label{eq:subnormal_minstable2}
\triv = G_{1,0} \triangleleft G_{1,1} \triangleleft \cdots \triangleleft G_{1,{r-1}} \leq G_{1,r} = G_1,
\end{equation*}
be a composition series for $(G_1,\alpha_1)$.  Then, setting $G_{2,j}$ to be the nub of the restriction of $\alpha_2$ to $\varphi_2\circ \varphi_1^{-1}(G_{1,j})$, 
\begin{equation*}
\label{eq:subnormal_minstable3}
\triv = G_{2,0} \triangleleft G_{2,1} \triangleleft \cdots \triangleleft G_{2,{r-1}} \leq G_{2,r} = G_2,
\end{equation*}
is a composition series for $(G_2,\alpha_2)$ and 
$$
(G_{1,j}/G_{1,j-1},\alpha_{1,j})\cong (G_{2,j}/G_{2,j-1},\alpha_{2,j})
$$ 
for each $j$. Example~\ref{ex:finite_centre} exhibits infinitely many pairs, $(G_n,\sigma\times\sigma)$, $n\in\mathbb{Z}$, having composition series of length $r=2$ and having composition factors $(C_2^{\mathbb{Z}},\sigma)$ and $(C_3^{\mathbb{Z}},\sigma)$ but which are not isomorphic because they are semidirect products with different actions of $(C_2^{\mathbb{Z}},\sigma)$ by automorphisms on $(C_3^{\mathbb{Z}},\sigma)$.
\end{remark}

\begin{example}
\label{ex:intersection_not_ergodic}
The group $G = C_2^{\mathbb{Z}}\times C_2^{\mathbb{Z}}$ with $\alpha = \sigma\times\sigma$ is clearly the internal direct product of $\alpha$-stable subgroups 
\begin{equation}
\label{eq:product1}
H_1 = C_2^{\mathbb{Z}}\times \triv\  \hbox{ and } \ H_2 = \triv\times C_2^{\mathbb{Z}}.
\end{equation}
However $(G,\alpha)$ is a product of $\alpha$-stable subgroups in many other ways.

Define $H_\varphi$ to be the subgroup 
$$
H_\varphi = \left\{ (\varphi(x), x)\in G \mid x\in C_2^{\mathbb{Z}}\right\},
$$ 
where $\varphi : C_2^{\mathbb{Z}}\to C_2^{\mathbb{Z}}$ is a surjective homomorphism. For example, $\varphi$ might be the homomorphism defined in Corollary~\ref{cor:not_complemented}. Then $H_\varphi$ is a closed $\alpha$-stable subgroup of~$G$ that isomorphic to $C_2^{\mathbb{Z}}$. 

If $(x,y)\in H_1\cap H_\varphi$, then $y=\ident$ because $(x,y)\in H_1$, which forces $x=\varphi(y) =\ident$ because $(x,y)\in H_\varphi$. Hence $H_1\cap H_\varphi = \triv$ and $G = H_1\times H_\varphi$. On the other hand, $(0,y)\in H_2\cap H_\varphi$ if and only if $y\in \ker\varphi$. Hence, although $G = H_2 H_\varphi$, it is not the direct product of $H_2$ and $H_\varphi$.

That $H_2\cap H_\varphi = \triv\times \ker\varphi$ shows that the intersection of topologically transitive $\alpha$-stable subgroups need not be topologically transitive. Note however that $\ker\varphi$ is a finite subgroup of $C_2^{\mathbb{Z}}$ and so $H_2\wedge H_\varphi = \triv$. 
\end{example}

\begin{example}
\label{ex:finite_centre}
The symmetric group $S_3$ is isomorphic to the semidirect product $C_3\rtimes_\chi C_2$ determined by the homomorphism $\chi : C_2\to \Aut(C_3)$ that sends $\bar{1}\in C_2$ to the automorphism $\bar{1}\mapsto \bar{2}$ of $C_3$. The group $S_3^{\mathbb{Z}}$ is consequently isomorphic to the semidirect product $C_3^{\mathbb{Z}}\rtimes_{\chi_0} C_2^{\mathbb{Z}}$ determined by the homomorphism $\chi_0 : C_2^{\mathbb{Z}}\to \Aut(C_3^{\mathbb{Z}})$ that sends $[\bar{1}]_s\in C_2^{\mathbb{Z}}$ to the automorphism such that
$$
\chi_0([\bar{1}]_s)(f)(r) =
\begin{cases}
\chi(f(r)), & \hbox{ if } r=s\\
f(r), & \hbox{ otherwise}
\end{cases}, \ \ (f\in C_3^{\mathbb{Z}}).
$$

Let $(\varprojlim C_2^{\mathbb{Z}},\tilde\sigma)$ be the pair defined in Example~\ref{ex:con_2-ways_not_dense} and $\varphi_1: {\varprojlim  C_2^{\mathbb{Z}}} \to C_2^{\mathbb{Z}}$ be the homomorphism also defined there. Then $\chi_0\circ \varphi_1 : {\varprojlim  C_2^{\mathbb{Z}}} \to \Aut(C_3^{\mathbb{Z}})$ is a homomorphism that commutes with the automorphisms $\sigma\in \Aut(C_3^{\mathbb{Z}})$ and $\tilde\sigma \in \Aut({\varprojlim  C_2^{\mathbb{Z}}})$. Put 
$$
G = C_3^{\mathbb{Z}}\rtimes_{\chi_0\circ \varphi_1} {\varprojlim  C_2^{\mathbb{Z}}}\hbox{ and }\alpha = (\sigma,\tilde\sigma).
$$
Identify $C_3^{\mathbb{Z}}$ and ${\varprojlim  C_2^{\mathbb{Z}}}$ with the obvious subgroups of $G$. Then $(G,\alpha)$ is a compact pair with $\bcg{G} = C_3^{\mathbb{Z}}$ and $Z(G) = \ker \varphi_1$. Hence $Z(G)\bcg{G}$ is a proper subgroup of~$G$.

The group $G$ may also be described directly as an inverse limit, as follows. For each $n\in \mathbb{N}$ let $A_n \subset \mathbb{Z}$ be the support of $\varphi^n([\bar{1}]_{0})$, where $\varphi : C_2^{\mathbb{Z}}\to C_2^{\mathbb{Z}}$ is defined by
$$
\varphi(g)(s) = g(s)-_2g(s+1).
$$
Define $\chi_n : C_2^{\mathbb{Z}} \to \Aut(C_3^{\mathbb{Z}})$ to be the homomorphism that sends $[\bar{1}]_s\in C_2^{\mathbb{Z}}$ to
\begin{equation}
\label{eq:centreexample}
\chi_n([\bar{1}]_s)(f)(r) =
\begin{cases}
\chi(f(r)), & \hbox{ if } r-s\in A_n\\
f(r), & \hbox{ otherwise} 
\end{cases}, \ \ (f\in C_3^{\mathbb{Z}})
\end{equation}
and put $G_n = C_3^{\mathbb{Z}}\rtimes_{\chi_n} C_2^{\mathbb{Z}}$. Then calculation shows that $Z(G_n)$ is the subgroup of $C_2^{\mathbb{Z}}$ consisting of $2^n$-periodic elements. Let $\varphi_{n+1,n} : G_{n+1}\to G_n$ be defined at $(f,g)$ in $C_3^{\mathbb{Z}}\rtimes_{\chi_n} C_2^{\mathbb{Z}}$ by
$$
\varphi_{n+1,n}(f,g) = (f,\phi(g)). 
$$
The definition of the set $A$ in (\ref{eq:centreexample}) implies that $\varphi_{n+1,n}$ is a homomorphism. Composing produces $\varphi_{m,n} : G_m\to G_n$ $(m\geq n)$ and an inverse system $(G_n, \varphi_{m,n})$. Then $G \cong {\varprojlim  (G_n,\varphi_{m,n})}$. 
\end{example}

\section{Topologically Transitive Automorphisms}
\label{sec:alphastable}

In his work on simple subgroups of automorphism groups of trees, \cite{Tits2}, J.~Tits introduced a property that he called (P). This property, which is defined in terms of the action on the tree, has been important in subsequent papers on automorphism groups of trees, see \cite[Section~3]{CapracedeMedts}, and has been extended to groups of automorphisms of complexes in \cite{HaglundPaulin}. One of the important consequences of this property is found in \cite[Lemme~4.3]{Tits2} and the next proposition is an abstract version of this result. In Tits' formulation of his Lemme~4.3, property~(P) asserts that the pair $(\nub{\alpha},\alpha)$ is isomorphic to a shift over an infinite group. The hypothesis of the proposition is therefore considerably weaker. 

\begin{proposition}
\label{prop:GenTits}
Let $(G,\alpha)$ be a compact pair with $\alpha$ topologically transitive. Then for every $k\in \ZZ\setminus\{0\}$ the map $\eta_k : x\mapsto x^{-1}\alpha^k(x)$ is surjective. 
\end{proposition}
\begin{proof}
The proof is in three steps, the first of which is the case when $(G,\alpha)$ is a shift. The argument here is essentially the same as that in \cite[Lemme~4.3]{Tits2}. Suppose that $(G,\alpha)\cong (F^{\mathbb{Z}},\sigma)$ and let $f\in F^{\mathbb{Z}}$. Assume, without loss of generality, that $k>0$ and define $x\in F^{\mathbb{Z}}$ by $x(n) = \ident_F$ if $n\in \{0,1,\dots,k-1\}$ and then define $x$ recursively by $x(n) =x(n+k) f(n)^{-1}$ if $n<0$ and $x(n) = x(n-k)f(n-k)$ if $n\geq k$. Then $f = \eta_k(x)$.
\par

Next suppose that $S\leq G$ is normal, $\alpha$-stable, closed and such that $(G/S,\alpha|^{G/S})$ is isomorphic to a shift. Suppose also that the map $\eta_k|_S : x\mapsto x^{-1}\alpha^{k}|_S(x)$ is surjective on~$S$. Let $f\in G$. Then, by the first paragraph, there is $x\in G$ such that $f = x^{-1}\alpha^k(x)$ modulo $S$. Hence $xf\alpha^k(x)^{-1}$ belongs to $S$ and, by assumption, there is $y\in S$ such that $xf\alpha^k(x)^{-1} = y^{-1}\alpha^k(y)$. Then $f = (yx)^{-1}\alpha^k(yx) = \eta_k(yx)$. An induction argument now shows that $\eta_k$ is surjective whenever $G$ has a composition series as in (\ref{eq:opennormal}), that is, whenever $(G,\alpha)$ has finite depth.
\par

Finally, suppose that $(G,\alpha)$ is topologically transitive and let $V$ be an open normal subgroup of $G$. Define $K_0 = \bigcap_{k\in{\ZZ}} \alpha^k(V)$, so that $G/K_0$ has finite depth. Then $\alpha$ corestricts to an automorphism of $G/K_0$ and it follows by the previous paragraph that the map $x\mapsto x^{-1}\alpha^k(x)$ is surjective on $G$ modulo $K_0$. In particular, the map ${\mathfrak q}_V\circ \eta_k$ is surjective on $G/V$, where ${\mathfrak q}_V$ denotes the quotient map $G\to G/V$. Since this holds for every compact open subgroup $V$, it follows that $\eta_k(G)$ is dense in $G$. On the other hand, $\eta_k$ is continuous and $G$ is compact, so that $\eta_k(G)$ is closed. Therefore $\eta_k$ is surjective. 
\end{proof}

The map $\eta_k$ is generally not injective  because there may be $k$-periodic elements of~$G$, that is, elements $x$ satisfying $\alpha^k(x) = x$. This occurs, for example, for $(F^{\mathbb{Z}},\sigma)$, where the set of periodic elements is dense. Moreover, it is shown in \cite[Theorem~5.7]{KSchmidt} that, if $G$ is abelian and $(G,\alpha)$ has finite depth, then the set of points periodic for $\alpha$ is dense in $G$. The set of periodic points may fail to be dense if $(G,\alpha)$ does not have finite depth, as seen in \cite[Examples~5.6]{KSchmidt}. For the case described in Example~\ref{ex:con_2-ways_not_dense}, the argument from \cite{KSchmidt} shows that the group ${\varprojlim  G_n}$ has no non-trivial $2^j$-periodic points for $\tilde\sigma$. 

The final result shows that an element of a, not necessarily compact, group that is invariant modulo $\nub{\alpha}$ is equal modulo $\nub{\alpha}$ to an invariant element.  
\begin{proposition}
\label{prop:constant}
Let $G$ be a totally disconnected, locally compact group and let $\alpha\in\Aut(G)$. Suppose that $x\in G$ satisfies $\alpha(x\nub{\alpha}) = x\nub{\alpha}$. Then there is $y\in x\nub{\alpha})$ such that $\alpha(y) = y$.
\end{proposition}
\begin{proof}
Let $h\in\nub{\alpha}$ be such that $\alpha(x) = xh$. Then for each $g\in \nub{\alpha}$ we have $\alpha(xg) = xh\alpha(g)$. Choose $g$ such that $g\alpha(g^{-1}) = h$ and put $y = xg$. 
\end{proof}

\end{document}